\documentclass[12pt]{amsart}

\usepackage[colorlinks=true, pdfstartview=FitV, linkcolor=blue, citecolor=blue, urlcolor=blue]{hyperref}

\usepackage{amssymb,amsmath, amscd}
\usepackage{times, verbatim}
\usepackage{graphicx}
\input xy
\xyoption{all}

\usepackage{mathrsfs}
\usepackage[mathcal]{euscript}
\usepackage{dsfont}
\usepackage[T1]{fontenc}

\usepackage{enumerate}
\usepackage{epsfig}
\usepackage{MnSymbol}
\RequirePackage{color}
\usepackage{mathrsfs}
\usepackage{tikz}
\usepackage{tikz-cd}

\usetikzlibrary{matrix,arrows}

\usepackage{diagrams}

\usepackage{amsfonts}
\usepackage{amscd}
\usepackage{mathrsfs}
\usepackage{amsmath}
\usepackage{enumerate}
\usepackage{epsfig}
\RequirePackage{color}
\usepackage{soul}

\usepackage{pgf}
\usepackage{tikz}
\usepackage{tikz-cd}

\usetikzlibrary{matrix,arrows}

\newtheorem{theorem}{Theorem}[section]
\newtheorem{Theorem}{Theorem}[section]
\newtheorem{lemma}[theorem]{Lemma}
\newtheorem{proposition}[theorem]{Proposition}

\newtheorem{corollary}[theorem]{Corollary}

\newtheorem{definition}[theorem]{Definition}

\newtheorem{example}[theorem]{Example}

\def\C{\mathbb C}

\def\Z{\mathbb Z}

\newcommand{\ZZ}{\mathbf{Z}}
\newcommand{\FF}{\mathbf{F}}
\newcommand{\RR}{\mathbf{R}}

\newcommand{\QQ}{\mathbf{Q}}

\newcommand{\A}{\mathbb{A}}
\newcommand{\PP}{\mathbb{P}}
\newcommand{\CC}{\mathbf{C}}

\newcommand{\fp}{{\mathfrak{p}}}

\newcommand{\qlbar}{\bar{\QQ}_{\ell}  }

\newcommand{\mc}[1]{\mathcal{#1}}
\newcommand{\mbf}[1]{\mathbf{#1}}
\newcommand{\mfr}[1]{\mathfrak{#1}}
\newcommand{\mr}[1]{\mathrm{#1}}
\newcommand{\an}{^{\mr{an}}}

 \newcommand{\slr}{\mr{SL}_2( \RR)}

  \newcommand{\slz}{\mr{SL}_2( \ZZ)}

\newcommand{\Hom}{\mathrm{Hom}}

\def\({\left(}
\def\){\right)}

\def\l{\lambda}
\def\eps{\varepsilon}
\def \ol{\overline}

\def\fp{\mathfrak p}

\newcommand*\HYPERskip{&}
\catcode`,\active
\newcommand*\pFq{
\begingroup
\catcode`\,\active
\def ,{\HYPERskip}%
\doHyper
}
\catcode`\,12
\def\doHyper#1#2#3#4#5{%
\, _{#1}F_{#2}\left[\begin{matrix}#3 \smallskip \\  #4\end{matrix} \; ; \; #5\right]%
\endgroup
}

\newcommand*\HYPER{&}
\catcode`,\active
\newcommand*\pFFq{
\begingroup
\catcode`\,\active
\def ,{\HYPER}%
\doHyperF
}
\catcode`\,12
\def\doHyperF#1#2#3#4#5{%
\, _{#1}{\mathbb F}_{#2}\left[\begin{matrix}#3 \smallskip \\  #4\end{matrix} \; ; \; #5\right]%
\endgroup
}

\catcode`,\active

\catcode`\,12

\newcommand*\HYPERpp{&}
\catcode`,\active
\newcommand*\pPPq{
\begingroup
\catcode`\,\active
\def ,{\HYPERpp}%
\doHyperFpp
}
\catcode`\,12
\def\doHyperFpp#1#2#3#4#5{%
\, _{#1}{\mathbb P}_{#2}\left[\begin{matrix}#3 \smallskip \\  #4\end{matrix} \; ; \; #5\right]%
\endgroup
}

\catcode`,\active

\catcode`\,12

\newcommand*\ApellFF{&}
\catcode`,\active
\newcommand*\FAFn{
\begingroup
\catcode`\,\active
\def ,{\ApellFF}%
\doApellFFp
}
\catcode`\,12
\def\doApellFFp#1#2#3#4#5{%
\, {\mathbb F}^{(#1)}_D\left[\begin{matrix}#2;&#3 \smallskip \\  {} & #4 \end{matrix} \; ; \;#5\right]%
\endgroup
}

\catcode`,\active

\catcode`\,12

\catcode`,\active

\catcode`\,12

\newcommand{\DM}{\mathbf{DA}}

\newcommand{\un}{\mathds{1}}

\newcommand{\Spec}{{\rm Spec}}

\begin{document}
\title{Transformations of hypergeometric motives}

\author{Jerome William Hoffman}

\address{Department of Mathematics \\
              Louisiana State University \\
              Baton Rouge, Louisiana 70803}

\author{Fang-Ting Tu}

\address{Department of Mathematics \\
              Louisiana State University \\
              Baton Rouge, Louisiana 70803}

\email{hoffman@math.lsu.edu}

\email{tu@math.lsu.edu}



\begin{abstract}
We consider algebraic transformations of hypergeometric functions from a geometric point of view.
Hypergeometric functions are shown to arise from the deRham realization of a hypergeometric motive. 
The $\ell$-adic realization of the motive gives rise to hypergeometric characters sums over finite fields. 
This helps to unify  and explain some recent results about transformations of hypergeometric 
character sums.

\end{abstract}

\maketitle

\section{Introduction.}
\label{S:intro}
The transformation theory of hypergeometric functions goes back at least to Gauss and Kummer.
A transformation  of hypergeometric functions is an identity

$$\pFq{2} {1}{a, b} {,c }{R(z)}  = C(z)\pFq{2} {1}{\alpha, \beta} {,\gamma} {S(z)}$$  
 for  rational 
functions $R(z), S(z)$ and an algebraic function $C(z)$. Kummer discovered many such identities by difficult calculations. Later 
Riemann deduced some of Kummer's identities by proving that 
the hypergeometric differential equation is the unique second order analytic differential equation
with three regular singularities at $z = 0, 1, \infty$ and with prescribed local monodromy  about these points. Following Katz, we say that the hypergeometric 
differential equation is {\it rigid}.

More recently, character sums over finite fields have been introduced which are analogs 
of hypergeometric functions, and many transformation identities have been discovered for them, 
analogous to the formulas known for the complex-analytic hypergeometric functions. The point of this 
paper is to explain that this is not an accident; both of these hypergeometric functions are manifestations (or realizations) of a 
{\it hypergeometric motive}. The complex-analytic
function reflects the Hodge-deRham realization, whereas the finite-field hypergeometric function 
reflects the $\ell$-adic \'etale realization. There are also $p$-adic crystalline realizations.  

The theme of this paper is that the transformations of hypergeometric functions are really geometric in origin-they
are transformations of motives. Thus, one gets transformation formulas in all three realizations. However, 
the assertion that all transformation identities of hypergeometrics are geometric in origin is really an imprecise conjecture. 
In certain cases 
it follows readily from standard conjectures (Hodge, Tate) about algebraic cycles. 
For GKZ systems, see \cite{Ter1}.
Nonetheless, we prove a weaker form of a consequence of this 
in theorem \ref{T:compare}. This allows one to deduce from a complex-analytic transformation
formula, a corresponding identity for finite-field hypergeometrics, up to an unspecified Galois 
twist.

In stating that all transformations arise from geometry, it should be understood that 
these geometric correspondences are very diverse. All $\phantom{}_2F_1$ hypergeometric motives 
occur in the cohomology of the families of curves 
\[
y^N = x^i (1-x)^j(1-\lambda x)^k
\]
for nonnegative integers $N, i, j, k$. In principle, any transformation can be expressed 
as a correspondence among various of these curves, but this is often unnatural. Hypergeometric 
motives occur in the cohomology of other families of algebraic varieties. 
Example:
recently, Yifan Yang and the second author of this paper
have discovered some transformations of hypergeometric equations using 
the theory of Shimura curves, \cite{TY1}. The motives are then families of abelian varieties 
with quaternion multiplication. These will be explored in more detail in a subsequent 
publication.

The word {\it hypergeometric} is used in a general sense in this paper. See section 
\ref{S:hypsheaf}. We do not consider {\it confluent} hypergeometrics. 
Finite field analogs of these are given e.g., by Kloosterman sums. Their $\ell$-adic sheaves
have wild ramification.
The Hodge-deRham story involves irregular Hodge theory, which has undergone 
a rapid development recently, see \cite{ESY}, \cite{FSY}, \cite{Yu}. In this paper, only tame ramification and regular singularities are permitted.

 
 A word on our use of the term {\it motivic}. Generally speaking, one expects to have a category 
 of motivic sheaves with the formalism of the 6 functors and realizations into various cohomology theories.
  We will explain in section 
 \ref{S:app3} the formalism we use. 
 Also, the term {\it hypergeometric motive} has already appeared and there is even a package 
 in Magma for computing with these. Those hypergeometric motives are special cases of the ones
 considered here. 
 
 An outline of this paper: In section \ref{S:DEmono} we recall the  relations between regular singular
 differential equations and monodromy. Section  \ref{S:hypsheaf} is a general discussion of sheaves attached to hypergeometric 
 functions. In section \ref{S:classical} this is specialized to the classical $_{2}F_1$ function. Section \ref{S:cyclo} 
 analyzes the cohomology of a family of curves relevant to this paper. In section \ref{S:hypmot} we give 
 the definition of hypergeometric motives used in this paper. In \ref{S:compare} we prove a theorem 
 that allows one to deduce a transformation formula for $\ell$-adic sheaves, knowing one for the 
 corresponding $\mc{D}$-module. Section \ref{S:charshv} reviews the formalism on $\ell$-adic sheaves in application 
 to character sums. Our main theorem \ref{T:main} is proved there. Section \ref{S:rigid} explains Katz's theory of rigid local systems. 
 This is specialized to Appell-Lauricella systems in \ref{S:AL}.
In sections \ref{S:exrig}, \ref{S:ATG}, \ref{S:exelliptic} examples are given of transformation formulas related respectively 
to rigidity, arithmetic triangle groups, and elliptic curves. In section \ref{S:picard} we discuss a transformation formula 
for an Appell-Lauricella system arising from the Picard family of curves. Appendices \ref{S:app1}, \ref{S:app3} explain the 
formalism of local systems: over $\CC$, $\ell$-adic, and motivic. 

We would like to thank Donu Arapura and Joseph Ayoub for information regarding the theory of motives. 

\section{Differential equations and monodromy}
\label{S:DEmono}
Riemann introduced 
the idea of monodromy into the study of analytic differential equations. 
Given  a representation of the fundamental group
\[
\rho : \pi _1 (\PP^1 (\CC) - \{ 0, 1, \infty\}, x) \to \mr{GL}_2(\CC)
\]
there is a unique second order rational differential equation with regular singular points 
 $z = 0, 1, \infty$ with the property that if $f_1(z), f_2(z)$ is a basis of holomorphic 
 solutions at $x$ then analytic continuation around a loop $\gamma \in \pi _1 (\PP^1 (\CC) - \{ 0, 1, \infty\}, x)$
 yields the linear transformation
 \[
  \begin{pmatrix} f_1\\ f_2 \end{pmatrix} \to \rho (\gamma) \begin{pmatrix} f_1\\ f_2 \end{pmatrix}.
 \]
Since 
\[
 \pi _1 (\PP^1 (\CC) - \{ 0, 1, \infty\}, x) = \langle \gamma _1, \gamma _2, \gamma _3\mid
 \gamma _1 \gamma _2 \gamma _3  = 1\rangle , 
\]
to give the monodromy representation is equivalent to giving two-by-two matrices
$\rho (\gamma _1) , \rho (\gamma _2) , \rho (\gamma _3)  $ such that 
$\rho (\gamma _1) \rho (\gamma _2) \rho (\gamma _3) = 1 $. These are well-defined up to 
simultaneous conjugation by an element of  $\mr{GL}_2(\CC)$. As Katz observes, Riemann proved 
a stronger result. Namely, it suffices to give the Jordan canonical forms of 
$\rho (\gamma _1) , \rho (\gamma _2) , \rho (\gamma _3)  $ to reconstruct the hypergeometric
differential equation. Actually Riemann only considered the case when these were semisimple, so 
equivalent to diagonal matrices, with eigenvalues $\exp (2\pi i \alpha )$ and  $\exp (2\pi i \alpha ' )$; 
he called the  $\alpha, \alpha '$ the exponents at the singular point. They are well-defined up to permutation and
adding $\ZZ$. This stronger property, that a differential equation is determined by the Jordan forms
of the monodromies at the singular points, is what Katz (\cite{Katz96}) calls {\it rigidity}. 

If $X$ is a nonsingular algebraic variety over $\CC$ and 
\[
\rho : \pi _1 (X^{\mr{an}},  x) \to \mr{GL}_n(\CC)
\]
is a representation, we get a local system {\sf  V} on the analytic space $X ^{\mr{an}}$. By the Riemann-Hilbert
correspondence, this is the solution sheaf to a differential equation, unique up to isomorphism,
\[
\nabla : \mc{V} \to \Omega^ 1 _{X/\C} \otimes _{\mc{O}_X} \mc{V}
\]
with regular singular points at infinity (see \cite{DelDE}).  A fundamental result asserts that the differential equations for the
periods of algebraic varieties have regular singular points with quasi-unipotent local monodromy.  This is due to 
Griffiths, Landman and Grothendieck. Arithmetic proofs of these results are given in \cite{Katz70i}. See also \cite{griff70}.

\section{Hypergeometric Sheaves}
 \label{S:hypsheaf}

The word {\it hypergeometric} will be understood in a generalized sense:
they include the $_p F_q$, the Pochhammer equations, Appell-Lauricella 
equations. The most general form of these are the GKZ (Gelfand-Kapranov-Zelevinski)
systems, \cite{GKZ}.

Generally speaking, a hypergeometric function is one that 

 \begin{itemize}
 \item[1.] has power-series expansions in special form: $\Gamma$-series;
 \item[2.] satisfies a (regular) holonomic system of differential equations;
  \item[3.] has Euler integral expressions;
  \item[4.] is attached to a motivic sheaf.
  \end{itemize}

Because of 4 above, we expect {\it realizations} of hypergeometric systems.
Let the motivic sheaf $\mathcal{H}$ be defined on $X/S$ where $X$ is a smooth $S$-scheme
with $S = O_F [1/N]$, $F$ = an algebraic number field, $O_F$ its ring of integers, 
$N \ge 1$ an integer. A typical case is 
\[
X = \PP^1 - \{  \text{a\ finite\ number\  of\  points}\}\ \ 
\mr{or} \ \  X = \PP^N - \{  \text{a\ finite\ number\  of\  hyperplanes}\}.
\] 
We expect 
\begin{itemize}
\item[\bf{Betti}.]
A Betti realization: $\mathcal{H} _{\sigma, \ZZ}$ on $X^{\text{an}}$, a local system of constructible
$\ZZ$-modules on the analytic space  $X^{\text{an}} = X^{\text{an}} _{\sigma}$, attached to each embedding 
$\sigma:  R \to \CC$.

\item[\bf{HdR}.] 
A Hodge-deRham realization: $\mathcal{H}_{dR}$ on $X$ which is a locally free sheaf 
in the Zariski topology
with an integrable connection 
\[
\nabla : \mathcal{H}_{dR} \to \mathcal{H}_{dR}\otimes _{\mc{O}_X} \Omega ^1 _{X /S}.
\]
For each embedding $\sigma : R \to \CC$, 
\[
\mathcal{H} _{\sigma, \CC} := \mathcal{H} _{\sigma, \ZZ} \otimes _{\ZZ} \CC= \mr{Ker} (\nabla _{\sigma}^{\text{an}}),
\]
the sheaf of analytic solutions to the algebraic differential 
equation $\nabla$. There is a comparison isomorphism
\[
\text{comp} _{\sigma}: \mathcal{H} _{\sigma, \ZZ} \otimes _{\ZZ}\CC \cong  \mathcal{H} ^{\text{an}}_{\sigma, dR} :=
 \mathcal{H}_{dR}\otimes _{\mc{O}_X }\mc{O} _{X _{\sigma}^{\text{an}}}.
\]
Written in a local flat frame for $ \mathcal{H} _{\sigma, \ZZ} $, the above isomorphism is given by a matrix
whose entries are analytic functions. These are the hypergeometric functions. They can be 
expressed as Euler integrals, and are  {\it periods} of these motives. Typically this structure extends to  a variation of Hodge 
structures, or are projections of these onto character eigenspaces. 

\item[\bf{$\ell$-adic}.]
For each good prime $\ell$, an $\ell$-adic realization: this is a lisse $\qlbar$-sheaf $\mathcal{H}_{\ell}$ on $X_{et}[1/\ell]$
whose Frobenius traces give a function
\[
X[1/\ell] (\FF _{q^n})\ni x \mapsto \mr{Tr} (\mr{Frob}_x \mid \mathcal{H}_{\ell, \bar{x}}) \in \qlbar.
\]
These functions are finite-field analogs of hypergeometric functions. This theory was developed 
principally by Katz; see his works \cite{Katz88}, \cite{Katz90}, \cite{Katz96}, \cite{Katz09}. For the 
$\ell$-adic version of GKZ see \cite{LF}.

\item[\bf{Crys}.] Finally there are 
$p$-adic crystalline realizations, and relations to $p$-adic Hodge theory. This 
is not as developed as the previous three. Recently, a Frobenius structure
has been established for GKZ systems, see forthcoming works of Fu/Wan/Zhang, \cite{FWZ}.
\end{itemize}

 The above are related by a series of compatibilities, which will not be written. See
 \cite{Huber95}.
 
Note that there are irregular differential equations of hypergeometric type, the confluent
hypergeometrics. The Hodge-deRham realizations then belong to irregular Hodge theory. 
They are related to character  
 sums involving additive as well as multiplicative characters of finite fields, and hence 
 their $\ell$-adic sheaves have wild ramification at infinity. 
In this note we will simplify the discussion by only considering tamely ramified sheaves, 
and only from the $\mathcal{D}$-module and $\ell$-adic point of view.

We will use the word {\it transformation} as follows. 
 This is an identity
 \[
 f^* \mc{M}_1 =  g^* \mc{M}_2 \otimes \mc{K}
  \]
  where $ \mc{M}_1$ is a hypergeometric sheaf on $X_1$,   
  $ \mc{M}_2$ is a hypergeometric sheaf on $X_2$, and  $\mc{K}$
  is a sheaf on $X$, in a diagram  
  \[ 
\begin{diagram}
&& X   & & \\
& f\ \ \ \  \ldTo\ & & \rdTo\ \ \ \ \ \  g\\
X_1 & & & & X_2 
\end{diagram}
\]

\section{The classical hypergeometric equation}
\label{S:classical}

See \cite[6.0, 6.1, 6.8.0]{Katz72}. 
For any scheme $T$, we denote by $\lambda$ the coordinate on $\mathbb{A}^1 _T$ and 
by $S_T$ the open set where $\lambda (1-\lambda)$ is invertible. Given any sections
$a, b, c, \in \Gamma (T, \mc{O}_T)$ we denote by $E(a, b, c)$ the free $\mc{O}_{S_T}$-module of rank
2 with basis $e_0$, $e_1$,  and integrable $T$-connection

\begin{align*}
\nabla \left (\frac{d}{d\lambda}\right )(e_0) = &e_1\\
\nabla \left (\frac{d}{d\lambda}\right )(e_1) = 
&-\frac{(c-(a+b+1)\lambda)}{\lambda (1-\lambda)} e_1 +\frac{ab}{\lambda (1-\lambda)} e_0.
\end{align*}
Horizontal sections of the dual of $E(a, b, c)$ over an open set $U \subset S_T$ can be identified with 
sections $f \in \Gamma (U, \mc{O}_U)$ which satisfy the differential equation
\[
{\lambda (1-\lambda)} \left (\frac{d}{d \lambda} \right )^2 f 
+ (c -(a+b+1)\lambda)\frac{df}{d \lambda} -ab f = 0.
\]
These hypergeometric equations are two-dimensional factors of the the cohomology of the
family of curves $y^N = x^a (x-1)^b(x-\lambda)^c$. In effect, the Euler integral representation

\[
F(\alpha, \beta; \gamma; \lambda)= \frac{\Gamma (\gamma)}{\Gamma (\beta)\Gamma(\beta - \gamma)}
\int _1 ^{\infty} x ^{\alpha - \gamma}(x-1)^{\gamma-\beta -1}(x-\lambda)^{-\alpha}dx
\]
shows that the solutions to the differential equation are given by periods of those curves.  

Given integers $N, a, b, c$ greater that zero. Let $Y(N;a, b, c)_{\lambda}$ be the nonsingular projective model of the 
affine curve in $(x, y)$-space defined by the equation  $y^N = x^a (x-1)^b(x-\lambda)^c$. 
This is a family of curves depending on the parameter $\lambda \neq 0, 1$. We consider 
the  family of curves 
 $$f: Y(N; a, b, c)\to  U := \mathbb{P}^1 -\{ 0, 1, \infty \}.$$
 We have the Gauss-Manin connection 
\[
\nabla: H^1 _{DR}(Y(N;a, b, c)/U)\to \Omega ^1 _{U/\CC}\otimes  _{\mc{O}_U}H^1 _{DR}(Y(N;a, b, c)/U).
\]
The following theorem gives the structure of this, at least in the generic fiber
$\mr{Spec}(\C (\lambda)) \hookrightarrow U$. Let
\[
X(N; a, b, c) = \mr{Spec} \ \CC (\lambda) [x, y, 1/y]/(y^N - x^a (x-1)^b(x-\lambda)^c).
\]
This is an open affine subset where $y$ is invertible. It is affine and smooth of relative 
dimension one over $\CC(\lambda)$. The map $(x, y)\to y$ is a finite 
\'etale covering of 
$$ \mathbb{A}^1_{\CC(\lambda)}  - \{0, 1, \lambda \}:= \mr{Spec}\  \CC (\lambda)[x, (x (x-1)(x-\lambda))^{-1}].$$
For any root of unity $\xi \in \mu _N$ there is an automorphism of $X(N; a, b, c) $ given by 
$(x, y) \mapsto (x, \xi y)$. This gives the Galois group of the covering 
$$\pi : X(N; a, b, c) \to\mathbb{A}^1_{\CC(\lambda)}  - \{0, 1, \lambda \}. $$ Note that 
the $dx/y^m$ defines an element in the character eigenspace
$ : H^1 _{DR}(X(N;a, b, c)/\CC(\lambda)) ^ {\chi (m)}$
where $\chi (t) (\xi) = \xi ^{-t}$ (the inverse of Katz's convention).

\begin{proposition}
(\cite[6.8.6]{Katz72}) Suppose that $N$ does not divide $a, b, c, a+b+c$. Then for any
integer $k\ge 1$ which is invertible modulo $N$ the map
\begin{align*}
e_0 &\mapsto \mr{the\  class\  of\ \ } \frac{dx}{y^k}\\
e_1 &\mapsto\nabla \left (  \frac{d}{d\lambda} \right ) ( \mr{the\  class\  of\ \ }\frac{dx}{y^k})
\end{align*}
induces an isomorphism
\[
E\left( \frac{kc}{N}, \frac{k(a+b+c)}{N}-1, \frac{k(a+c)}{N}   \right ) \cong
H^1 _{DR}(X(N;a, b, c)/\CC(\lambda)) ^ {\chi (k)}.
\]
\end{proposition}
Note that this gives only the part of the cohomology belonging to primitive characters
modulo $N$. 
  
The local system $R^1 f_* \CC$ on $U^{an} $ underlies a polarized variation of Hodge structures
of weight 1. The $H^1 _{DR}(X(N;a, b, c)/\CC(\lambda)) ^ {\chi (k)}$ are the modules with connection 
that correspond to the rank 2 local system  $(R^1 f_* \CC) ^{\chi (k)}$. Note that 
 $(R^1 f_* \CC) ^{\chi (k)}$ does not correspond to a variation of Hodge structures unless 
 the character $\chi(k)$ is real. Nonetheless there are period mappings attached to this situation
 (see \cite{DK}).

To consider the $\ell$-adic realization, let $S = \mr{Spec}(R_N)$, where
$R_N = \ZZ[\zeta _N, 1/N]$. The eigenspaces  $(R^1 f_* \qlbar) ^{\chi (k)}$
then give the $\ell$-adic realization, where $f: Y(N;a, b, c) \to U$ is as before 
but now as schemes over $S$. The \'etale topology is understood here.

\section{Cohomology of cycloelliptic curves}
\label{S:cyclo}
\subsection{}
\label{SS:cyclo1}
A cycloelliptic curve is the projective nonsingular model of 
\[
 y^N = x^i (1-x)^j(1-\lambda _1 x)^{k_1}...(1-\lambda _r x)^{k_r}.
\]
A more symmetric numbering is to take 
\[
X_{\mathbf{\lambda}}^{[N; \mathbf{i}]}= X: y^N = \prod _{j=0}^{r+1}(x-\lambda _j) ^{i_j} , \quad \mbf{i}  = (i_0, ..., i_{r+1}) .
\]
At first, we examine this over $\CC$, with fixed  values of the parameters $\lambda_1, ..., \lambda _r$, and we use $X$ to denote 
the corresponding Riemann surface. 
The natural projection $p : X \to P$ sending $(x, y)\mapsto x$ makes 
$X$ into a Galois $\mu _N$-branched covering of $P = \mathbb{P}^1_{x}$. We define the action as 
$y \mapsto \zeta _N y$, $\zeta _N = \exp(2 \pi i/N)$. The branching occurs over a subset of
 \[
 S = \{\lambda_{0}, ..., \lambda _{r+1}, \infty \} \supset S_0 =  \{\lambda_{0}, ..., \lambda _{r+1} \}.
 \]
In our set-up, the branching will be over all of $S$. We let
$T = p^{-1}(S) \subset X$, which is the subset of $X$ where $y \neq 0, \infty$. We let $X^{\circ}= X - T$, $P^{\circ}= P - S$. These are affine smooth curves
and the projection $p ^{\circ} :X^{\circ}\to P^{\circ}$ is an \'etale $\mu _N$-covering.
We have a Cartesian square
\[
\begin{CD}
X ^{\circ}@>j>> X\\
@Vp^{\circ}VV @VpVV\\
P ^{\circ}@>j'>> P\\
\end{CD}
\]

The cohomology decomposes
\[
H^1(X, \CC) = \bigoplus _{\chi: \mu _N \to \CC^{\times}}      H^1(X, \CC)  ^{\chi}
\]
where the sum is over the characters $\chi$ and the superscript refers to the $\chi$-eigenspace. One can replace
the coefficients $\CC$ by a smaller field, e.g, $K_N = \QQ(\mu _N)$. The sheaf sequence 
\[
\begin{CD}
0@>>> j_{!} \CC_{X^{\circ}} @>>> \CC_X @>>> \CC_T @>>> 0
\end{CD}
\]
gives
\[
\begin{CD}
...@>>>H^i_{c} (X^{\circ}, \CC) @>>>H^i (X, \CC)  @>>>H^i (T, \CC)  @>>> ...
\end{CD}
\]
which shows that 
\[
H^1_{c} (X^{\circ}, \CC) = H^1 (X, \CC) \oplus \CC^{\# T -1}
\]
where the first summand is pure of weight 1, and the second factor is pure of weight 0 (of 
Hodge type $(0, 0)$). This decomposes into eigenpaces for $\chi \in\widehat {\mu _N}$. 

 Projecting the above sheaf sequence by $p$ we get
 \[
 \begin{CD}
 0@>>>p_{\ast} j_{!} \CC_{X^{\circ}} @>>>p_{\ast} \CC_X @>>> p_{\ast}\CC_T @>>> 0\\
 && @V=VV @V=VV @V=VV\\
 0 @>>> \displaystyle{\bigoplus _{\chi \in \widehat {\mu _N}} j' _{!}L_{\chi}}@>>>  
 \displaystyle{\bigoplus _{\chi \in \widehat {\mu _N}}\tilde{L}_{\chi}}@>>> 
 \displaystyle{\bigoplus _{\chi \in \widehat {\mu _N}}} (p_{\ast}\CC_T)^{\chi} @>>>0
 \end{CD}
 \]
For each character $\chi$, $L_{\chi}$ is a rank 1 $\CC$-local system on $P ^{\circ}$, 
$\tilde{L}_{\chi}$ is a constructible sheaf of $\CC$-vector spaces on $P$, and 
$j': P^{\circ}\to P$ is the inclusion. 
We have
\[
(p_{\ast}\CC_T)^{\chi}  = \bigoplus _{s\in S} (p_{\ast}\CC_T) _s^{\chi}. 
\]
By Leray, we get
\[
H^1(X, \CC) ^{\chi} = H^1(P, \tilde{L}_{\chi}), \quad  H_c^1(X^{\circ}, \CC) ^{\chi} = H_c^1(P^{\circ}, L_{\chi}). 
\]
By choosing a root of unity $\zeta _N = \exp (2 \pi i/N)$ we can identify $\widehat {\mu _N}= \ZZ/N$.
Then the local system $L_{\chi}$ belonging to the character $\chi_k (\zeta _N) = \zeta _N^k$ is the 
subsheaf
\[
L_{\chi} = \CC y^k \subset \mathcal{O} ^{\mr{hol}} _{P ^{\circ}}
\]
where $y$ is any branch of $\sqrt[N]{ \prod _{j=0}^{r+1}(x-\lambda _j) ^{i_j} }$.

\begin{theorem}
\label{T:cohcyclo}
Assume that for each $j$, $i_j \nequiv 0 $ mod $N$ and that $i_0 + ...+i_{r+1}\nequiv 0$ mod $N$. Then 
for each primitive character $\chi \in \widehat {\mu _N} ^{\mr{prim}}$,
\[
H_c^1(P^{\circ}, L_{\chi}) = H_c^1(X^{\circ}, \CC) ^{\chi} = H^1(X, \CC) ^{\chi}=H^1(P, \tilde{L}_{\chi}).
\]
The above space has dimension $r+1$.
\end{theorem}
\begin{proof}
To show the first claim, it suffices to show that for all $\chi \in \widehat {\mu _N} ^{\mr{prim}}$, and for each 
$s \in S$, we have 
\[
 (p_{\ast}\CC_T) _s^{\chi} = 0.
\]
To prove the second claim, the Euler characteristic 
\[
\chi_c(P^{\circ}, L_{\chi} ) = 2 -\# S = 2 - (r+3) = -(r+1),
\]
since  $L_{\chi}$ is a local system of rank 1. Under these hypotheses, we will see that each $L_{\chi}$ is a nontrivial
local system, and therefore $H_c^i(P^{\circ}, L_{\chi}) =0$ for $i=0, 2$. Note that
 $H^0_c = 0$ because $P^{\circ}$ is not compact; by duality $H^2 _c  (L)= H^0 (L^*)$, and the latter is zero
 because $L$ is nontrivial. Thus $\dim H_c^1(P^{\circ}, L_{\chi}) = r+1$.
 
 For each divisor $d$ of $N$, let  $ \widehat {\mu _{d}} \subset  \widehat {\mu _{N}} $ be the subset of those 
 characters that factor $\mu_N \to \mu _{d}\to \CC ^{\times}$, where the first map is $\zeta \mapsto\zeta ^{N/d}$
 The primitive characters are those that do not factor for any divisor $d< N$.
 For each $s \in S$ we let $d_s = \gcd (N, i_s)$ if $s$ is a finite point, and for $s = \infty$, 
 $d_{\infty} = \gcd (N,  i_0 + ...+i_{r+1})$. By our hypothesis, 
  each $d_s < N$. We will show, that as $\mu _N$ representation
 \[
  (p_{\ast}\CC_T) _s = \mr{Ind} _{N/d_s} ^N (1) = \sum _{\chi \in  \widehat {\mu _{d_s}} } \chi.
 \]
 That being so, no primitive character appears in any of these, so $(p_{\ast}\CC_T) _s ^{\chi}=0$ for primitive 
 characters.
 
 In more detail: the equation for the curve can be written $y^N  = \prod _{s\in S_0} t_s ^{i_s}$, where
 $t_s = x-\lambda _s$ is a local parameter at $s\in S_0$. In the local ring at $s$ this is $y^N = (\mr{unit}) t_s ^{i_s}$, 
 so to analyze the ramification above $s\in S_0$, we can consider the equation $y^N = t_s ^{i_s}$. For the ramification at 
 $\infty$, we use the parameter $t_{\infty}= 1/x$, and the local equation is 
  $y^N = t_{\infty}^{i_0 + ...+i_{r+1}}$. 
  
 Writing, for each $s\in S_0$, $N = N_s d_s, i_s = j_s d_s$; 
  $N = N_{\infty} d_{\infty}, \sum _{s \in S_0}i_s = j_{\infty} d_{\infty}$, we see from the 
  factorization
  \[
  y^N - t_s ^{i_s}  = (y^{N_s})^{d_s} - (t_s ^{i_s} )^{d_s} =\prod_{\omega\in \mu_{d_s}}(y^{N_s} -\omega t_s ^{j_s})
  \]
  that the fiber of $p$ above $s\in S$ consists of $d_s$ points, each totally ramified of degree $N_s$. This is because 
  $\gcd(j_s, N_s) =1 $, and each local curve $y^{N_s} -\omega t_s ^{j_s}=0$ is isomorphic to a disk, say by the map
  $u \mapsto (y, t)=  (u^{j_s} \omega ^{1/N_s}, u^{N_s}$). If $t \in p^{-1}(s)$ the fiber  $\CC _t$ is stabilized by 
  the subgroup $\mu _{N/d_s}$, and since the action of $\mu _N$ is transitive on  $p^{-1}(s)$ we see that 
  as a $\mu_N$ representation, $(p_{\ast}\CC _T)_s$ is the induced module  $\mr{Ind} _{N/d_s} ^N (1)$, as claimed.
  
Each $L_{\chi}$ is a nontrivial local system, if $\chi= \chi _k$ is primitive. One can see this by considering the local monodromy 
around any point $s\in S$.  Analytic continuation of $y^k$ around $s$ is given by the character
$(\zeta _N )^{k i_s}\neq 1$ if $\gcd(k, N) = 1$, since $i_s \nequiv 0$ mod $N$.
\end{proof}

The above theorem is valid for any algebraically closed base-field $k$, where analytic cohomology is replaced by 
\'etale cohomology, that is, for $H_c^1(X^{\circ}, \bar{\QQ}_{\ell}) ^{\chi} $, provided that the characteristic of $k$
is prime to $N\ell$. The proof is exactly the same (replace disks by the Henselian local rings). The only nontrivial 
point to observe is that all the local systems are tame. 

Here is a picture:
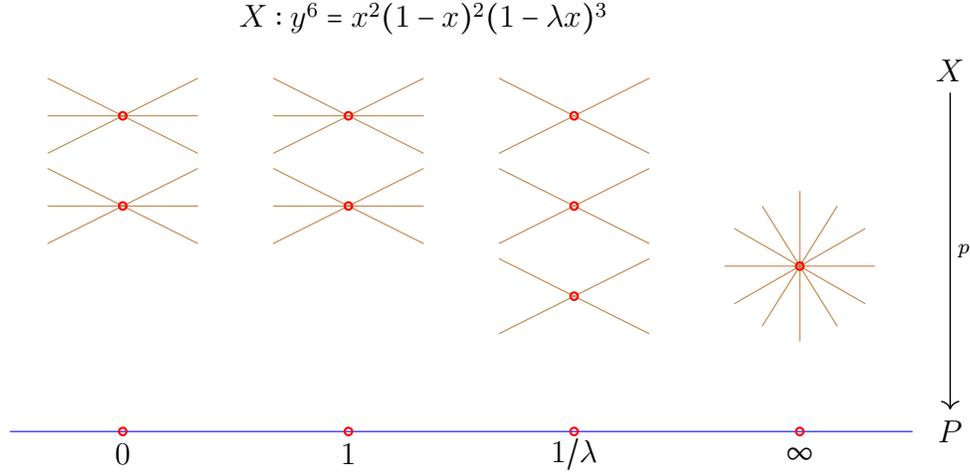
\[
\begin{tikzpicture}

\node (C) at  (5.5, 5.5) {$ X: y^6=x^2(1-x)^2(1-\lambda x)^3$};

\draw[brown] (3.5,4.7) --(5.5,3.7);
\draw[brown] (3.5,4.2) --(5.5, 4.2);
\draw[brown] (3.5,3.7) --(5.5,4.7);
\draw [red, thick] (4.5,4.2) circle [radius=0.05];

\draw[brown] (3.5,3.5) --(5.5,2.5);
\draw[brown] (3.5,3) --(5.5, 3);
\draw[brown] (3.5,2.5) --(5.5,3.5);
\draw [red, thick] (4.5,3) circle [radius=0.05];

\draw[brown] (0.5,4.7) --(2.5,3.7);
 \draw[brown] (0.5,4.2) --(2.5, 4.2);
\draw[brown] (0.5,3.7) --(2.5,4.7);
 \draw [red, thick] (1.5,4.2) circle [radius=0.05];
 
 \draw[brown] (0.5,3.5) --(2.5,2.5);
 \draw[brown] (0.5,3) --(2.5, 3);
\draw[brown] (0.5,2.5) --(2.5,3.5);
 \draw [red, thick] (1.5,3) circle [radius=0.05];
 

\draw[blue] (0,0) --(12, 0);
 \draw [red, thick] (1.5,0) circle [radius=0.05];
 \node (P0) at (1.5, -.3) {$0$};
 \draw [red, thick] (4.5,0) circle [radius=0.05];
 \node (R0) at  (4.5, -.3) {$1$};
 \draw [red, thick] (7.5,0) circle [radius=0.05];
 \node (R0) at  (7.5, -.3) {$1/\lambda$};
 \draw [red, thick] (10.5,0) circle [radius=0.05];
 \node (R0) at  (10.5, -.3) {$\infty$};

 \draw[brown] (6.5,4.7) --(8.5,3.7);
 \draw[brown] (6.5,3.7) --(8.5,4.7);
 \draw [red, thick] (7.5,4.2) circle [radius=0.05];
 
 \draw[brown] (6.5,3.5) --(8.5,2.5);
 \draw[brown] (6.5,2.5) --(8.5,3.5);
 \draw [red, thick] (7.5,3) circle [radius=0.05];
 
  \draw[brown] (6.5,2.3) --(8.5,1.3);
  \draw[brown] (6.5,1.3) --(8.5,2.3);
  \draw [red, thick] (7.5,1.8) circle [radius=0.05];

  \draw[brown] (9.63,2.7) --(11.37,1.7);
  \draw[brown] (9.5,2.2) --(11.5,2.2);
  \draw[brown] (9.63,1.7) --(11.37,2.7);
   \draw[brown] (10.5,1.2) --(10.5,3.2);
    \draw[brown] (10,1.4) --(11,3);
     \draw[brown] (10,3) --(11,1.4);
  \draw [red, thick] (10.5,2.2) circle [radius=0.05];
  
  

   \node (S0) at  (12.5, 4.8) {$ X$};
   \node (T0) at  (12.5, 0) {$P$};
   \path[commutative diagrams/.cd, every arrow, every label]
   (S0) edge node {$p$} (T0);
\end{tikzpicture}
\]

\subsection{}
\label{SS:cyclo2}
Now we consider the dependence of the curves on the parameters $\lambda = \{\lambda _0, ...,\lambda_{r+1} \}$.
Let 
\[
D(\lambda) = \prod_{i< j} (\lambda _i- \lambda _j), \quad h(x) =  \prod_{i= 0}^{{r+1}} (x -\lambda _i). 
\]
Let 
\[
R_N = \ZZ[\mu _N, 1/N],  S_N = R_N[\lambda, D(\lambda)^{-1}], 
T_N = S_N[x, h^{-1}].
\]
Let 
\[
U=   \Spec(S_N) =  \mathbb{A} ^{r+1} _{R_N} - \{D(\lambda) = 0\}, \mr{coordinates\ } \lambda.
\]
Let 
\[
P^{\circ}_U =   \Spec(T_N) = \mathbb{A}^1 _{U} -\{h = 0 \}, \mr{coordinate\ } x.
\] 
There is an evident $R_N$-morphism $u: P^{\circ}_U  \to U$. 
Let 
\[
 X^{\circ} = \Spec T_N[y, y^{-1}]/(y^N - \prod_{j = 0}^{r+1} (x - \lambda _j)^{i_j}).
\] 

The natural map $p^{\circ}: X^{\circ}\to P^{\circ}_U$ sending $(x, y)\to x$ is an \'etale $\mu _N$-covering. 
The affine curve  $X^{\circ} $ is the open subset of the projective, nonsingular model
$X = X^{[N; \mathbf{i}]}$ where $y \neq 0$. The composite map $\alpha:=u \circ p^{\circ}:  X^{\circ} \to U $
sends the curve to its corresponding $\lambda$ value. We omit reference to the ring of constants $R_N$ when it is clear. 
Here is a picture ($S = \{ 0, 1, 1/\lambda, \infty\}$):

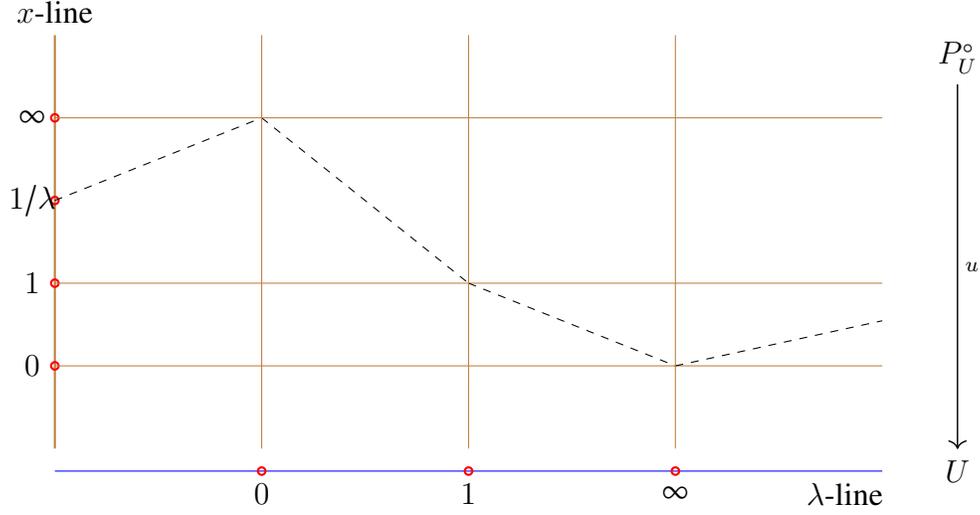
\[
\begin{tikzpicture}

\draw[blue] (0,0) --(11, 0);
\draw [red, thick] (2.75,0) circle [radius=0.05];
\node (R0) at (2.75, -.3) {$0$};
\draw [red, thick] (5.5,0) circle [radius=0.05];
\node (R1) at  (5.5, -.3) {$1$};
\draw [red, thick] (8.25,0) circle [radius=0.05];
\node (Ro) at  (8.25, -.3) {$\infty$};
\node (P0) at (10.5, -0.3) {$\lambda$-line};

\draw[brown, thick] (0,0.3) --(0,5.8);
\draw [red, thick] (0,1.4) circle [radius=0.05];
\node (P0) at (-0.3, 1.4) {$0$};

\draw [red, thick] (0,2.5) circle [radius=0.05];
\node (P0) at (-0.3, 2.5) {$1$};

\draw [red, thick] (0,3.6) circle [radius=0.05];
\node (P0) at (-0.3, 3.6) {$1/\lambda$};

\draw [red, thick] (0,4.7) circle [radius=0.05];
\node (P0) at (-0.3, 4.7) {$\infty$};

\node (P0) at (0, 6.1) {$x$-line};

\draw[brown] (0,1.4) --(11,1.4);
\draw[brown] (0,2.5) --(11,2.5);
\draw[dashed] (0,3.6) -- (2.75, 4.7);
\draw[dashed]  (2.75, 4.7)--(5.5,2.5);
\draw[dashed]  (5.5, 2.5)--(8.25,1.4);
\draw[dashed]  (8.25, 1.4)--(11,2.0);
\draw[brown] (0,4.7) --(11,4.7);

\draw[brown] (2.75,0.3) --(2.75,5.8);
\draw[brown] (5.5,0.3) --(5.5,5.8);
\draw[brown] (8.25,0.3) --(8.25,5.8);

\node (S0) at  (12, 5.5) {$P^{\circ}_U$};
\node (T0) at  (12, 0) {$ U$};
\path[commutative diagrams/.cd, every arrow, every label]
(S0) edge node {$u$} (T0);

\end{tikzpicture}
\]

Let $f(x)  = f_{\mbf{i}} (x) =   \prod_{j = 0}^{r+1} (x - \lambda _j)^{i_j} \in T_N$.
This defines a morphism 
$f: P^{\circ}_U \to \mbf{G}_m$, and
we have a Cartesian diagram (schemes over $R_N$)
\[
\begin{CD}
X^{\circ} @>g>> \mbf{G}_m\\
@Vp^{\circ} VV  @V N VV\\
P^{\circ}_U @> f >> \mbf{G}_m
\end{CD}
\]
 We get
\[
f^{\ast} N_{\ast }\bar{\QQ} _{\ell} = \bigoplus  _{\chi}f^{\ast} K(\chi)_{\ell} =p^{\circ} _{\ast} g^{\ast} \bar{\QQ} _{\ell}  =  
p^{\circ} _{\ast}  \bar{\QQ} _{\ell}, 
\]
where the sum is over all the characters $\chi : \mu _N \to \bar{\QQ} _{\ell} ^{\times}$, $K(\chi)$ is the Kummer sheaf, see
appendix  \ref{S:app3}. 
The lisse sheaf on $U$ given by $R^1 \alpha _{!}\,  \bar{\QQ} _{\ell}$ gives the cohomologies of the curves in each fiber, 
viz., 
\[
(R^1 \alpha _{!}\,  \bar{\QQ} _{\ell}) _ {\bar {\lambda}} = H^1 _c (X^{\circ} _ {\bar {\lambda}},  \bar{\QQ} _{\ell})
\]
for each geometric point $\bar{\lambda} $ on $U$. Since $p^{\circ}$ is finite, we have
\[
R^1 \alpha _{!}\,  \bar{\QQ} _{\ell} = R^1 u _{!}\, p^{\circ} _{\ast} \bar{\QQ} _{\ell} 
=  \bigoplus  _{\chi}  R^1 u _{!}\, f^{\ast} K(\chi)_{\ell}.
\]
This justifies our taking $ R^1u _{!}\, f^{\ast} K(\chi)_{\ell}$ as the $\ell$-adic realization of a hypergeometric sheaf.

\begin{definition}
\label{D:hypmot}
Let $K_N = \QQ (\zeta _N)$ and  $\chi : \mu _N \to K_N ^{\times}$ be a primitive character. Assume that
$N$ does not divide any $i_j$ or $i_0+...+i_{r+1}$. 
In the notations above, we define
\[
\mathcal{P}[\mathbf{i}/N, \chi] := 
R u _{!}\,f _{\mbf{i}}^{\ast} K(\chi)
\]
in $\mathbf{DA}(U, K_N)$.
\end{definition}
 We ought to define this as $R^1 u _{!}\,f _{\mbf{i}}^{\ast} K(\chi)$, 
but this requires a $t$-structure on our motives, only conjecturally available. In our case, 
 $R^i u _{!}\,f _{\mbf{i}}^{\ast} K(\chi) _{\ell} = 0$ for $i \neq 1$, so this is harmless. One can also make use of other theories 
 of motives that do have $t$-structures, e.g., Nori motives.

Symbolically, we can write this as 
\[
\mr{Jac}(X/U)^{\chi}, 
\]
where $\mr{Jac}(X/U) \to U$ is the abelian scheme of the relative Jacobians of the curves in the fibers. 
Note that 
\[
\mr{Jac}(X/U)^{\mr{prim}} = \bigoplus _{\chi \in \mu_N ^{\mr{prim}}}\mr{Jac}(X/U)^{\chi},
\]
is meaningful as an abelian scheme up to isogeny, but the individual summands only make sense as motives.

\section{Hypergeometric motives}
\label{S:hypmot}

For the main properties of the fundamental group, see \cite{SGA1}.
Let $F$ be a finite extension field of $\QQ$. $R = O_F[1/N]$ the localization of the ring of 
integers of $F$ for an integer $N\ge1$. Let $S = \mr{Spec}(R)$. We let 
$\eta = \mr{Spec}(F)$, the generic point of $S$, and $\bar{\eta} = \mr{Spec}(\bar{F})$ for an 
algebraic closure of $F$. Let $U/S$ be an irreducible separated scheme, smooth and of finite type over $S$, 
with geometrically connected fibers. 

As a first approximation, by a {\it motivic sheaf} on $U$ we mean the following: 
\[
\mc{H} = (\mc{H}_B,   \mc{H}_{DR}, \mc{H}_{\ell}   )
\]
where 
\begin{itemize}
\item[1.] $\mc{H}_B$ is a local system of finite-dimensional $\QQ$-vector spaces
on $U^{\mr{an}}$.
\item[2.]  $\mc{H}_{DR}$ is a  locally free $\mc{O}_U$-module with an integrable connection
$\nabla: \mc{H}_{DR} \to \Omega ^1 _U \otimes _{\mc{O}_U} \mc{H}_{DR} $
with regular singular points at infinity (i.e., relative to a smooth compactification of $U$.)
\item[3.] For each prime number $\ell$ prime to the residual characteristics of $U$, 
$\mc{H}_{\ell}$ is a lisse $\bar{\QQ}_{\ell}$-sheaf on $U _{et}$.
\end{itemize}
These are subject to a number of properties, of which we single out the comparison 
isomorphisms:
\begin{itemize}
\item[1.] There is an isomorphism $\mc{H}_B\otimes _{\QQ} \CC \cong\mr{Ker}( \nabla ^{\mr{an}})$
where $\mr{Ker}( \nabla ^{\mr{an}})$ is the sheaf  of solutions of the analytic differential equation 
attached to $\nabla$.

\item[2.] For each prime $\ell$ not dividing $N$, there is an isomorphism of fields 
$\iota: \qlbar\cong \CC$ and an isomorphism 
\[
(\mc{H}_{\ell}) ^{\mr{an}} \cong \mc{H}_{B}\otimes _{\QQ} \CC
\]
 of $\CC$-local systems on $U^{\mr{an}}$.
\end{itemize}
 Comparison 2 has the following meaning: The lisse $\qlbar$-sheaf $\mc{H}_{\ell}$ on $U$ is equivalent to a representation 
 \[
 \rho: \pi _1 (U, \bar{\eta})  \to \mr{GL}(V)
 \]
for a finite-dimensional $\qlbar$-vector space $V$, where the left-hand side is the \'etale fundamental 
group. There is a canonical map $\pi (U^{\mr{an}}, u) \to  \pi _1 (U, \bar{\eta}) $ where 
the left-hand side is the usual fundamental group, and where $u \in U(\CC)$ is a base-point
which lies over the base-point $\bar{\eta}$. That is, if $\bar{\eta} : \mr{Spec}(K) \to U$
is the base-point attached to an algebraically closed field, then $u$ is the composite 
$\mr{Spec}(\CC)\to  \mr{Spec}(K)  \to U$ for an embedding $K \subset \CC$. By Riemann's existence
theorem, $\pi _1 (U, \bar{\eta})$ is the profinite completion of $\pi (U^{\mr{an}}, u) $, and in particular, 
the image is dense in the profinite topology of the target.

 Via this 
isomorphism, we obtain 
\[
\rho ^{\mr{an}}:\pi (U^{\mr{an}}, u)  \to   \mr{GL}(V) \cong  \mr{GL}(V_{\CC}), 
\quad V_{\CC} := V \otimes_ {\qlbar,  \iota } \CC,
\]
where the last isomorphism comes from $\iota$. This defines the $\CC$-local system 
$(\mc{H}_{\ell}) ^{\mr{an}}$. Statement 2 is that this is isomorphic to the 
$\CC$-local system $\mc{H}_{B}\otimes _{\QQ} \CC$. This $\CC$-local system is equivalent
by statement 1 to the differential equation $\mc{H}_{DR}$.

In Section \ref{S:app3} we describe more precisely the triangulated categories of motivic sheaves.

\section{Comparison Theorem}
\label{S:compare}
The main result of this section is to show that a transformation identity among hypergeometric
differential equations implies a similar one among finite field hypergeometric functions, up to twisting
by a Galois character.  As before, let $F$ be a finite extension field of $\QQ$. $R = O_F[1/N]$ the localization of the ring of 
integers of $F$ for an integer $N\ge1$. Let $S = \mr{Spec}(R)$. We let 
$\eta = \mr{Spec}(F)$, the generic point of $S$, and $\bar{\eta} = \mr{Spec}(\bar{F})$ for an 
algebraic closure of $F$. Let $U/S$ be an irreducible separated scheme, smooth and of finite type over $S$, 
with geometrically connected fibers. 
We can choose a geometric generic point $\bar{\xi} : \Spec (\overline{F(U)}) \to U$ which 
lies over $\bar{\eta}$, where $F(U) $ is the function field of $U$. We let $U_{\eta}$  and $U_{\bar{\eta}}$  
be the schemes over $\Spec{F}$ and $\Spec{\bar{F}}$ obtained from $U$ by base-change. 

We consider two geometrically irreducible lisse $\bar{\QQ}_{\ell}$-
sheaves (for the \'etale topology) $\mathscr{F}$, $\mathscr{G}$ on $U$. These are equivalent to 
two $\ell$-adic representations 
\[
\rho_{\mathscr{F}}: \pi _1(U, \bar{\xi}) \longrightarrow
\mr{GL}(V), \quad
\rho_{\mathscr{G}}: \pi _1(U, \bar{\xi}) \longrightarrow
\mr{GL}(W)
\]
for finite-dimensional $\bar{\QQ}_{\ell}$-vector spaces $V$, $W$. Geometrically irreducible means:
the restrictions 
\[
\rho_{\mathscr{F}}\mid U_{\bar{\eta}}, \ \ \rho_{\mathscr{G}}\mid U_{\bar{\eta}}
\]
of $ \pi _1(U_{\bar{\eta}}, \bar{\xi}) $ are irreducible. Note that we have 
 a surjective homomorphism
\[
 \pi _1(U_{\eta}, \bar{\xi}) \to  \pi _1(U, \bar{\xi})
\]
 and an exact sequence
\[
\begin{CD}
0 @>>>\pi _1(U_{\bar{\eta}}, \bar{\xi}) @>>>\pi _1(U_{\eta}, \bar{\xi}) @>>>
\mr{Gal}(\bar{F}/F)@>>>0.
\end{CD}
\]
Now let $\varphi :R \to \CC$ be an embedding. We obtain a scheme
$U_{\varphi, \CC}$ over $\CC$. This also defines an analytic space
$U_{\varphi}^{\mr{an}} := U_{\varphi}(\CC) $. Since $\varphi$ will be fixed, 
we will drop it from the notation. There is a canonical map
\[ 
\pi_1(U^{\mr{an}}, u) \to \pi_1(U_{\CC}, u)
\]
(left-hand side: topological fundamental group; right-hand side, the \'etale fundamental group)
which identifies the right-hand side with the profinite completion of left-hand side (Riemann's existence
theorem). In particular, this map has dense image. Here we can take
$u$ to be the geometric $u: \mr{Spec}(\CC) \to U$ 
\[
\mr{Spec}(\CC) \to \mr{Spec}(\overline{F(U)}) = \bar{\xi} \to  U
\]
where the first arrow is induced by some embedding $\bar{\varphi}: \overline{F(U)} \to \CC$ which extends $\varphi$.
It is known that there are isomorphisms $\pi_1(U_{\CC}, u) =\pi_1(U_{\bar{\eta}}, u) =\pi_1(U_{\bar{\eta}}, \bar{\xi})$ induced
by $\bar{\varphi}$. The first holds because $F$ has characteristic 0; the second is a change in base-point.
Choose an isomorphism $\iota : \bar{\QQ}_{\ell} \cong \CC$. The theorem that follows will not depend 
on this artificial choice.

Composing all these, we get representations 
\[
\rho ^{\mr{an}}_{\mathscr{F}} : \pi_1(U^{\mr{an}}, u) \to \pi_1(U_{\CC}, u) =  \pi_1(U_{\bar{\eta}}, \bar{\xi})
\overset {\rho _{\mathscr{F}}}{\longrightarrow} \mr{GL}(V) \cong \mr{GL}(V_{\CC})
\]
where $V_{\CC} = V \otimes _{\bar{\QQ}_{\ell}, \iota}\CC$. We get a similar story for  $\rho ^{\mr{an}}_{\mathscr{G}}$.
We let ${\sf F} $ and ${\sf G}$ be the $\CC$-local systems on $U^{\mr{an}}$ that arise from these 
representations of the fundamental group. Also $\mathcal{D} ({\sf F})$ and  $\mathcal{D} ({\sf G})$ 
the regular holonomic $\mathcal{D}$-modules (=connections with regular singular points) corresponding 
to these by the Riemann-Hilbert correspondence.

\begin{theorem}
\label{T:compare}
Under these assumptions (and $\mathscr{F}$, $\mathscr{G}$ geometrically irreducible), if 
the local systems ${\sf F} $ and ${\sf G}$ on  $U^{\mr{an}}$ are isomorphic (equivalently
if the $\mathcal{D}$-modules $\mathcal{D} ({\sf F})$ and  $\mathcal{D} ({\sf G})$  are isomorphic), 
then there is a continuous character $\chi : \mr{Gal}(\bar{F}/F)\to \bar{\QQ}_{\ell}^{\times}$, such that $\mathscr{G}_{\eta} \cong \mathscr{F}_{\eta} \otimes \chi.$
\end{theorem}
\begin{proof}
There is a matrix $M : V_{\CC}\to W _{\CC} $ that intertwines 
the representation of $\pi_1(U^{\mr{an}}, u)$ given by ${\sf F} $ and ${\sf G}$. Then
$\iota ^{-1}(M) : V \to W$ is a matrix that intertwines the representations of  $\pi_1(U^{\mr{an}}, u)$ in 
$\mr{GL}(V)$ and $\mr{GL}(W)$. 
But $\pi_1(U^{\mr{an}}, u)$ has dense image in  $\pi_1(U_{\bar{\eta}}, \bar{\xi})$ and the 
representations given by $\mathscr{G}$ and $\mathscr{F}$ on $V$, $W$ are continuous. Thus
by continuity $\iota ^{-1}(M) $ will intertwine those representations. Therefore the representations
\[
\rho_{\mathscr{F}}\mid U_{\bar{\eta}}, \ \ \rho_{\mathscr{G}}\mid U_{\bar{\eta}}
\]
of $ \pi _1(U_{\bar{\eta}}, \bar{\eta}) $ are isomorphic. From the exact sequence above, and 
the fact that these representations are isomorphic we get, by the lemma below, a character 
$\chi : \mr{Gal}(\bar{F}/F)\to \bar{\QQ}_{\ell}^{\times}$ and an isomorphism
\[
(\rho_{\mathscr{G}}\mid U_{\eta}) = (\rho_{\mathscr{F}}\mid U_{\eta})\otimes \chi
\]
as representations of $ \pi _1(U_{\eta}, \bar{\xi}) $. 
 
 \end{proof}

The following is well-known. 

\begin{lemma}
\label{L:ext}
Given an exact sequence of groups 
\[
\begin{CD}
0 @>>>  H @>a>> G @>b>> G/H@>>>0
\end{CD}
\]
and two finite-dimensional representations $\rho : G \to \mr{GL}(V)$ and
$\sigma : G \to \mr{GL}(W)$ where $V, W$ are vector spaces over an
algebraically closed field $k$.
Suppose that $\rho \mid H$ and $\sigma \mid H$ are irreducible and isomorphic. 
Then there is a character $\chi : G/H \to \mr{GL}_1(k) = k ^{\times}$, such that 
$\sigma = \rho \otimes \chi := \rho \otimes (\chi\circ b)  $. If $k$ is a topological field and $\rho, \sigma $ are continuous 
representations, then $\chi $  is a continuous character. 
\end{lemma}


We can give a stronger version  of this theorem if we assume in addition that $U/S$ has a section, and that 
$U/S$ is the complement in $Z/S$ of a divisor with normal crossings $D/S$, where $Z/S$ is proper
and smooth. We also assume that $\ell$ is invertible on $U$ and $S$. Under those assumptions, then we have an exact sequence
 (\cite[Ch. XIII, Prop. 4.3, and Examples 4.4]{SGA1})
\[
\begin{CD}
0 @>>>\pi _1 ^{\mathbb{L}}(U_{\bar{\eta}}, \bar{\xi}) @>>>\pi _1 '( U, \bar{\xi}) @>>>
\pi _1 (S, \bar{\eta})@>>>0.
\end{CD}
\]
Here $\mathbb{L}$ is a set of primes invertible on $S$. $\pi _1 ^{\mathbb{L}}(U_{\bar{\eta}}, \bar{\xi})$
is the pro-$\mathbb{L}$- quotient of $\pi _1 (U_{\bar{\eta}}, \bar{\xi})$. If $K$ is the kernel of the canonical homomorphism
$\pi _1 ( U, \bar{\xi}) \to \pi _1 (S, \bar{\eta})$ and $N\subset K$ is the smallest normal subgroup such that 
$K/N$ is a pro-$\mathbb{L}$-group, then $N \subset \pi _1 ( U, \bar{\xi}) $ is a normal subgroup, and 
we denote $\pi _1 '( U, \bar{\xi})  = \pi _1 '( U, \bar{\xi}) /N$. If we assume that the representations
$\rho_{\mathscr{F}}$ and $\rho_{\mathscr{G}}$ factor through $\pi _1 '( U, \bar{\xi}) $ and are 
geometrically irreducible, we can conclude that there exists a continuous character 
$\chi: \pi _1 (S, \bar{\eta})\to \bar{\QQ}^{\times}_{\ell}$ such that  $\mathscr{G} \cong \mathscr{F} \otimes \chi$
on $U$.

\section{Character sums and $\ell$-adic sheaves}
\label{S:charshv}
\subsection{\rm{}}
\label{SS:charshv1}
Now let $X_0$ be a scheme separated and of finite type over a finite field $\FF _q$, 
with $X = X_0 \otimes _{\FF_q}\bar{\FF}_q$.
If $x \in |X|$ is a closed point, then the residue field $k(x)$ is a finite extension 
of $\FF_q$  whose degree  we denote by $\deg(x)$, so $k(x)$ has
$q^{\deg (x)}$ elements. If $\varphi \in \mr{Gal}(\bar{\FF}_q/\FF_q)$,  $\varphi (x) = x^q$ is the Frobenius substitution, the induced action 
of $\varphi $ on  $X(\bar{\FF}_q)$ coincides with the action of the Frobenius morphism
$F : X \to X$ on $|X|$: this is the morphism that sends the point with coordinates 
$x$ to the point with coordinates $x^q$.  The fixed points of iterates 
of the Frobenius can be identified with $\FF_{q^n}$-rational points,
$X^{F^n} = X_0 (\FF _{q^n})$.

Let $\ell$ be a prime number with $\ell \nmid q$. Given any constructible 
$\qlbar$-sheaf $\mc{F}$ on $X_0$, the fiber $\mc{F}_{\bar{x}}$ in a geometric 
point $\bar{x}$ over any $x\in X$ is a finite dimensional $\qlbar$-vector space with a continuous action 
of $\mr{Gal} (k(\bar{x}) /k(x))$. If  $X_0 $ is geometrically irreducible with generic point 
$\eta$, then to give a lisse 
$\mc{F}$ on $X_0$ is equivalent to give a continuous representation
\[
\rho : \pi _1(X_0, \bar{\eta}) \to \mr{GL} (\mc{F}_{ \bar{\eta}})\sim \mr{GL}_n (\qlbar).
\]
If $x$ is a closed point, there are automorphisms
$\mr{Frob}_x := \mr{Frob}_{q^{\deg (x)}}$ of the fibers $\mc{F}_{\bar{x}}$. Let
$K_0 \in D^b _c (X_0, \bar{\QQ}_{\ell})$ and we define $K := K_0 \otimes_{\FF_q} \bar{\FF}_q\in D^b _c (X, \bar{\QQ}_{\ell})$.
Then we have automorphisms 
$\mr{Frob}_x $ on the fibers of the constructible $\qlbar$-sheaves $\underline {H}^i (K _0)_{\bar{x}}$. By definition,
\[
\mr{Tr}(\mr{Frob}_x \mid K_0) := \sum _{i \in \ZZ} (-1)^i \mr{Tr}(\mr{Frob}_x\mid \underline {H}^i (K_0)_{\bar{x}}).
\]
We have the Grothendieck-Lefschetz fixed-point formula:
\[
\sum _{x \in X^F} \mr{Tr}(\mr{Frob}_x \mid K _0) = 
\sum_{i \in \ZZ} (-1)^i \mr{Tr} (F_q^* \mid \mr{H}^i _c (X, K)),
\]
where $F = F_q$ is the Frobenius endomorphism of $X$.  

For all $n \ge 1$, we can consider the 
functions on $X_0(\FF_{q^n})$ given by 
\[
X_0(\FF_{q^n})\ni x \mapsto       t_{K, n}(x):= \mr{Tr}(\mr{Frob}_x \mid K\otimes _{\FF_{q}}  \FF_{q^n} ).
\] 
These functions are well-defined on the Grothendieck 
group $K_0 (X_0, \qlbar)$ of constructible $\qlbar$-sheaves on $X_0$. This latter group 
is the free abelian group on the simple perverse sheaves on $X_0$. 
It follows from the \v{C}ebotarev density theorem that the trace functions 
determine the image of $K$ in that group in  the sense that if $K, L\in  D^b _c (X_0, \bar{\QQ}_{\ell})$
and for all $n \ge 1$ the trace functions are equal, i.e., $\forall x \in X_0(\FF_{q^n})$, 
$t_{K, n}(x) = t_{L, n}(x)$, 
then $K$ and $L$ give the same class in $K_0 (X_0, \qlbar)$. In particular
$K = L$ if $K$ and $L$ are simple perverse. For all of this, see the first sections of \cite{Lau}.

\subsection{\rm{}}
\label{SS:charshv2}

As Deligne  pointed out (\cite[Ch. 6]{SGA4.5}), many character sums over finite fields can be interpreted 
as functions 
\[
X_0 (\FF _q)\ni x \mapsto t_K (x) =  \mr{Tr}(\mr{Frob}_x \mid K) 
\]
for suitable $K$. This is so for the various hypergeometric character sums considered by 
many authors (see for instance \cite{BCM}, \cite{MF},  \cite{LF}, \cite{FLRST},  
\cite{Gr}, \cite{Katz90}). 

Let $\psi : \FF_q \to \qlbar ^{\times}$ be a nontrivial additive character. There is a lisse 
$\qlbar$-sheaf of rank one $\mc{L}_{\psi}$ on $\mbf{G}_a = \A^1$
(Artin-Schreier sheaf), with the property that 
$\mr{Tr}(\mr{Frob}_x |\mc{L}_{\psi, \bar{x}}) = \psi (\mr{Tr}_{\FF _{q^d} /\FF_q} (x))$ if 
$d = \deg(x)$. Let $\chi : \FF_q  ^{\times}\to \qlbar ^{\times}$ be a multiplicative character. There is a lisse 
$\qlbar$-sheaf of rank one $\mc{L}_{\chi}$ on $\mbf{G}_m = \A^1 -\{ 0\}$
(Kummer sheaf), with the property that 
$\mr{Tr}(\mr{Frob}_x |\mc{L}_{\chi, \bar{x}}) = \chi (\mr{N}_{\FF _{q^d} /\FF_q}(x)) $ if 
$d = \deg(x)$. If $f : S \to \mbf{G}_a $ (resp., $f : S \to \mbf{G}_m $) is a morphism, 
then we define $\mc{L}_{\psi (f)}= f^*\mc{L}_{\psi}$ on $S$ (resp., $\mc{L}_{\chi (f)}= f^*\mc{L}_{\chi}$ ). 

For the sheaf $\mc{L}_{\psi}\otimes \mc{L}_{\chi}$ on $\mbf{G}_m$ we have that 
$\mr{H}^i _c  (\mbf{G}_m, \mc{L}_{\psi}\otimes \mc{L}_{\chi})$ is zero unless $i = 1$; when 
$i = 1$ it is one-dimensional, and by the trace formula above, $F^*$  on $H^1$ is 
multiplication by the Gauss sum 
\[
-g(\psi; \chi) = -\sum_{x \in \FF _q ^{\times}} \psi (x) \chi(x).
\]

\subsection{\rm{}}
\label{SS:charshv3}
 Now let $\alpha _1, ..., \alpha _a$ and $\beta _1, ..., \beta _b$ be two disjoint unordered lists
 of multiplicative characters of $\FF_q^{\times}$, at least one of $a$, $b$ nonzero. We allow some of the 
 characters to be trivial, and we allow repetition in the lists. Then Katz \cite[Ch. 8]{Katz90} 
 defines a  geometrically irreducible $\qlbar$-sheaf (a hypergeometric sheaf of type $(a, b)$)
\[
\mc{H} (\psi; \alpha _i ;  \beta _j) \text{\ on \ } \mbf{G}_m / \FF_q
\]
with the property that if $E/\FF_q$ is a finite extension field, and $t \in \mbf{G}_m (E) = E^{\times}$, 
\begin{align*}
 \mr{Tr} (\mr{Frob }_t | \mc{H} (\psi; \alpha _i ;  \beta _j)_{\bar{t}}) &=(-1)^{a+b-1}\sum _{V(n, m, t)(E)} \psi_E (\sum _i x_i - \sum _j y_j )  \prod_{i} \alpha _{i, E} (x_i)
 \prod_{j} \bar{\beta }_{j, E} (y_j)\\
 &=(-1)^{a+b-1}\sum _{\Lambda \in   \widehat{E^{\times}}} \bar{\Lambda}(t)
\prod _i g(\psi _E, \Lambda \alpha _{i, E})\prod _j g(\bar{\psi} _E, \bar{\Lambda} \bar{\beta} _{i, E})
\end{align*}
where in the last expression, the sum ranges over the multiplicative 
characters of $E^{\times}$; this is the Fourier expansion of the first expression in terms 
of the characters of $E^{\times}$.  $\psi_E  = \psi\circ \mr{Tr}_{E /\FF_q} $,  $\chi _E = \chi\circ \mr{N}_{E /\FF_q} $, and the first
sum is over the 
$E$-rational points of the variety 
\[
V(n, m, t): \\ \ \  \prod_{i=1}^a x_i = t \prod_{j=1}^b y_j.
\] 
Without loss of generality, we can take $a \ge b$. Then this sheaf is of rank $a$ and is
pure of weight $a+b -1$. It is lisse on $\mbf{G}_m$ if $a \neq b$; if $a = b$ it is lisse on 
$\mbf{G}_m -\{ 1\}$, but tame at $\{ 0, 1, \infty\}$. The local monodromies of these sheaves 
are determined explicitly in terms of the characters $\alpha _i, \beta _j $ and their multiplicities.
When $a> b$ the sheaves $\mc{H} (\psi; \alpha _i ;  \beta _j)$ have wild ramification at $\infty$. 
The cohomology of these sheaves is given by
\[
\mr{H}^i _c (\mbf{G}_m \otimes_{\FF_q}   \bar{\FF}_q,   \mc{H} (\psi; \alpha _i ;  \beta _j)    ) = 0, 
\quad\text{unless \ }i = 1; \text{\ when\ } i = 1, \text{the dimension is }1,
\]
and 
\[
 \mr{Tr} (\mr{Frob }_E  \mid \mr{H}^1 _c (\mbf{G}_m \otimes_{\FF_q}   \bar{\FF}_q, 
 \mc{L} _{\Lambda} \otimes \mc{H} (\psi; \alpha _i ;  \beta _j)    )  = \prod _i (-g(\psi _E, \Lambda \alpha _{i, E}))\prod _j (-g(\bar{\psi} _E, \bar{\Lambda} \bar{\beta} _{i, E})).
\]

Also useful are the determinant formulas; see \cite[Theorem 8.12.2]{Katz90}. These hypergeometric sheaves are rigid 
in the sense to be recalled below. When the list of the $\beta _i$ is empty, we get the Kloosterman sums/sheaves
explored in detail in \cite{Katz88}.

In \cite{BCM} a variant of the above hypergeometric sum is introduced, where the summation 
is taken over the subvariety $W(n, m, t)\subset V(n, m, t)$ where
$\sum _i x_i = \sum _j y_j$. This has the advantage of removing the additive character $\psi$
and hence the wild ramification in the hypergeometric sheaf, hence giving objects that 
can be defined over the integers, not just over finite fields. 

When $a=b$ it is shown  in \cite[Section 4]{Katz09}  that there is a canonical twist 
$\mc{H}^{\mr{can}}( \alpha _i ;  \beta _j)=\mc{H} (\psi; \alpha _i ;  \beta _j)\otimes \Phi$,
by a $\qlbar$-valued character 
$\Phi $ of $\mr{Gal}(\bar{\FF}_q/\FF_q)$, which is independent of $\psi$. This is given by 
$\Phi = (1/A) ^{\mr{deg}}$ where 

\[
A = \prod_i g(\psi; \alpha_i ) g(\bar{\psi}, \bar{\beta}_i) =
\prod_i g(\psi; \alpha_i ) g(\psi, \bar{\beta}_i)\beta _i (-1).
\]
The character sum for $\mc{H}^{\mr{can}}( \alpha _i ;  \beta _j)$ then involves 
a twist by a product of Jacobi sums:
\[
J(\FF_q; \mu , \nu) = \sum _{x \in \FF_q ^{\times}}  \mu (x) \nu(1-x)
\]
 ($\mu (0) =0$, if $\mu \ne 1$; $\mu (0) =1$, if $\mu = 1$ ).
When $a=b$, $\mc{H}^{\mr{can}}( \alpha _i ;  \beta _j)$ globalizes 
to give an object over the integers in a suitable cyclotomic number field, with the Jacobi sum globalizing to 
a Hecke character. 

 For the canonical twist, we have 
\[
 \mr{Tr} (\mr{Frob }_E  \mid \mr{H}^1 _c (\mbf{G}_m \otimes_{\FF_q}   \bar{\FF}_q, 
 \mc{H} ^{\mr{can}} (\alpha _i ;  \beta _j)    )  = -1.
\]

\subsection{\rm{}}
\label{SS:charshv4}

When $a=b=2$ these are the finite field analogs of the $\phantom {}_2 F_1$.
This can be seen as follows: First, for $a=b =1$ and $t \in \FF_q ^{\times}$ an easy calculation shows that 
\[
 \mr{Tr} (\mr{Frob }_t | \mc{H} (\psi; \alpha ;  \beta )_{\bar{t}}) =  -g(\psi; \alpha \bar{\beta} )\alpha (t) (\beta/\alpha)(t-1)
\]
Using $J(\mu, \nu)g(\psi; \mu\nu) = g(\psi; \mu)g(\psi; \nu)$ if $\mu\nu \neq 1$, we see that the sum associated to the
canonical twist of this is 
\begin{equation}
\label{E:charshv1}
 \mr{Tr} (\mr{Frob }_t | \mc{H} ^{\mr{can}} ( \alpha ;  \beta )_{\bar{t}}) = 
\frac{1}{-J(\alpha, \beta /\alpha)}\alpha (t) (\beta/\alpha)(1-t).
\end{equation}

For $a=b =2$, 
\[
\mr{Tr} (\mr{Frob }_t | \mc{H} (\psi; \alpha_1, \alpha _2 ;  \beta_1, \beta _2 )_{\bar{t}}) =
-\sum _{\frac{x_1x_2}{y_1y_2}=t} \psi (x_1)\psi(x_2)\bar{\psi}(y_1)\bar{\psi}(y_2)
\alpha_1(x_1)\alpha_2(x_2)\bar{\beta}_1(y_1)\bar{\beta}_2(y_2)
\]
which can be rewritten as 
\[
-\sum _{(t_1, t_2), t_1t_2=t}
\left (\sum _{\frac{x_1}{y_1}=t_1} \psi (x_1-y_1)
\alpha_1(x_1) \bar{\beta _1 } (y_1)\right )
\left (\sum _{\frac{x_2}{y_2}=t_2} \psi (x_2-y_2)
\alpha_2(x_2) \bar{\beta _2} (y_2)\right ).
\]
The inner sums are just cases of $a=b=1$. This reduces to 
\[
C \sum _{ t_1t_2=t}
\alpha_1 (t_1) (\beta_1/\alpha_1)(t_1-1)\alpha_2 (t_2) (\beta_2/\alpha_2)(t_2-1), 
\quad C = -g(\psi; \alpha_1 \bar{\beta}_1 ) g(\psi; \alpha_2 \bar{\beta}_2 ).
\]
Replacing $t_2 = t/t_1$, this simplifies to
\[
C\alpha _2 (t) \sum _{ t_1} \lambda(t_1) \mu(t_1-1) 
\nu(t-t_1),\quad 
\lambda = \alpha _1 \beta _2 ^{-1}, \mu = \beta _1 \alpha _1 ^{-1},
\nu = \beta _2 \alpha _2 ^{-1},
\]
which is indeed a finite field analog of $\phantom {}_2 F_1$.

If 
$\chi : \mu _N \to \qlbar ^{\times}$ is a primitive character, and 
$\lambda = \chi ^a, \mu = \chi ^b,
\nu = \chi ^c$ the above sum becomes essentially 
$
\sum _{ t_1} \chi(t_1 ^a (t_1-1)^b  (t-t_1)^c).
$
These sums occur when counting the $\FF_q$-rational points on the smooth 
projective model of the curve
\[
y^N = x^a (x-1)^b (t-x)^c.
\]
Because 
\[
 \mc{H}^{\mr{can}} (\alpha_1, \alpha _2 ;  \beta_1, \beta _2 )[1] = 
  \mc{H}^{\mr{can}} (\alpha_1;  \beta_1)[1] \ast _{!} \mc{H}^{\mr{can}} ( \alpha _2 ;   \beta _2 )[1],
\]
(convolution) a calculation similar to the above gives 
\begin{lemma}
\label{L:charshv1}
\[
 \mr{Tr} (\mr{Frob }_t | \mc{H}^{\mr{can}} (\alpha_1, \alpha _2 ;  \beta_1, \beta _2 )_{\bar{t}}) =
 A\alpha _2 (t) \sum _{ t_1} \lambda(t_1) \mu(1-t_1) 
\nu(t-t_1),\quad 
\lambda = \alpha _1 \beta _2 ^{-1}, \mu = \beta _1 \alpha _1 ^{-1},
\nu = \beta _2 \alpha _2 ^{-1},
 \]
 where $A = -(\beta_2 \bar{\alpha}_2)(-1)
 [J(\alpha_1, \beta _1 \bar{\alpha}_1)  J(\alpha_2, \beta _2 \bar{\alpha}_2)    ]^{-1} $.
\end{lemma}
 
 \begin{corollary}
 \label{C:charshv1}
 \[
 \mr{Tr} (\mr{Frob }_1 | \mc{H}^{\mr{can}} (\alpha_1, \alpha _2 ;  \beta_1, \beta _2 )_{\bar{1}}) 
 = -\frac{\beta _2 \bar{\alpha}_2(-1) 
 J(\alpha _1 \bar{\beta}_2, \beta _1 \beta_2\bar{\alpha}_1\bar{\alpha}_2)}
 {J(\alpha _1, \beta_1 \bar{\alpha}_1)J(\alpha _ 2, \beta_2 \bar{\alpha}_2)}.
\]
  \end{corollary}
The sheaf  $\mc{H}^{\mr{can}} (\alpha_1, \alpha _2 ;  \beta_1, \beta _2 )$ is a 
rank 2 local system on  $U = \mbf{G}_m -\{ 1\}$ but the stalk at $t=1$ has rank 1. 
This is because the dimension of the invariants under inertia $I(1)$ is one-dimensional. 
In fact, the local monodromy at $1$ is a pseudoreflection of determinant 
$\beta _1 \beta _2/\alpha _1 \alpha _2$, and every pseudoreflection has a codimension one 
space of invariants.

 It is sometimes convenient to think of the sheaves 
 $\mc{H}^{\mr{can}} (\alpha_i;  \beta _j )$ as living on all of 
$ \mathbb{P}^1$. They will always be understood as
$j_{\ast} ( \mc{H}^{\mr{can}} (\alpha_i;  \beta _j )\mid U)  $ where 
$j : \mbf{G}_m -\{ 1\} \to \mathbb{P}^1$. It will be useful to calculate 
the Frobenius traces at the other singular points, viz., $0, \infty$, when these
stalks are nonzero.

 We remark:
 \begin{proposition}
 \label{P:monograph}
If $\psi$ is the canonical additive character character $x\mapsto e^{2 \pi i /p \cdot \mbox{tr(x)}}$ of $\FF_q$, we have  
\begin{align*}
    \mr{Tr} (\mr{Frob }_t | \mc{H} (\psi; \alpha_1, \alpha _2 ;  \beta_1, \beta _2 )_{\bar{t}}) =& \alpha_2\beta_2(-1)\alpha_2(t) \cdot C\cdot \pPPq 21{\alpha_2\ol{\beta_2}& \alpha_2\ol{\beta_1} }{&\alpha_2\ol{\alpha_1}}{t}\\
=&  \alpha_2\beta_2(-1)\alpha_2(t) \cdot C\cdot J(\FF_q;\alpha_2\ol{\beta_1} , \beta_1\ol{\alpha_1}) \pFFq 21{\alpha_2\ol{\beta_2}& \alpha_2\ol{\beta_1} }{&\alpha_2\ol{\alpha_1}}{t},
\end{align*}
the finite field ${}_{n+1}\mathbb P_{n}$- and  ${}_{n+1}\mathbb F_{n}$-functions in the monograph \cite[Chapter 4]{FLRST}.  Since  $\alpha _1, \alpha _2$ and $\beta _1,\beta _2$ are disjoint unordered lists of multiplicative characters, we can rewrite the character sum as 
$$
    \mr{Tr} (\mr{Frob }_t | \mc{H} (\psi; \alpha_1, \alpha _2 ;  \beta_1, \beta _2 )_{\bar{t}}) =
 \alpha_2\beta_2(-1)\alpha_2(t) \cdot C\cdot J(\FF_q;\alpha_2\ol{\beta_1} , \beta_1\ol{\alpha_1}) \pFFq 21{ \alpha_2\ol{\beta_1} &\alpha_2\ol{\beta_2} }{&\alpha_2\ol{\alpha_1}}{t}.
$$

More generally, for $a=b=n$,  if $\psi$ is the canonical additive character character $x\mapsto e^{2 \pi i /p \cdot \mbox{tr(x)}}$ of $\FF_q$, we have  
\begin{align*}
    \mr{Tr} (\mr{Frob }_t | \mc{H} (\psi; \alpha ;  \beta )_{\bar{t}}) = \alpha_n\beta_n(-1)\alpha_n(t) \cdot C\cdot  {}_{n+1}\mathbb P_n
    \left[\begin{matrix} \alpha_n\ol{\beta_n} & \alpha_n\ol{\beta_1}& \ldots & \alpha_n\ol{\beta_{n-1}} \smallskip  \\ & \alpha_n\ol{\alpha_{1}} & \ldots & \alpha_n\ol{\alpha_{n-1}}\end{matrix} \; ; \; t\right],
\end{align*}
the finite field ${}_{n+1}\mathbb P_{n}$-function in \cite{FLRST}, where $C=-\displaystyle \prod_{i=1}^{n-1}  g(\psi; \alpha_i \ol{\beta_i} ) $. 

\end{proposition}

\subsection{}
\label{SS:charshv5}
In chapter  \cite[Chapter 5]{FLRST} it is shown how, given rational numbers $a_i,b_j,\l\in \QQ$, one can attach
 a collection of hypergeometric functions over finite residue fields $\mathbf{F}_{\fp}$ (varying in $\fp$)
$$\pPPq{n+1}{n}{\iota_\fp(a_1)&\iota_\fp(a_2)&\cdots&\iota_\fp(a_{n+1})}{&\iota_\fp(b_1)&\cdots &\iota_\fp(b_n)}{\l;q(\fp)},$$
where $\fp$  runs through all unramified prime ideals of $\QQ(\zeta_N)$ with $N$ being the least positive common denominator of all $a_i$ and $b_j$. The symbols $\iota_\fp(a_1)$ etc. represent characters of the finite field $\mathbf{F}_{\fp}$ with 
$q(\fp)$ elements. The Frobenius traces of the $\ell$-adic realization of the hypergeometric motives can be expressed in terms of these 
functions. This is done in  \cite[Theorem 1.1, Chapter 6.2]{FLRST}.

We will generalize this result. Recall the notation from section \ref{SS:cyclo2}
We let $ |U| $be the set of closed points of $U$ (= maximal ideals of $S_N$).
Let $x \in |U|$ be a closed point. This lies over a prime ideal $\mfr{p} \subset R_N$. We let
$\mathbf{F}_q = R_N/\mathfrak{p}$, where
$q = N\mathfrak{p} \equiv 1$ mod $N$.
Then 
$\kappa (x)=  \mathcal{O}_{U, x}/\mathfrak{m}_{x} = \mathbf{F}_{q^d} $
is a finite extension of $\mathbf{F}_q$, $d = \deg (x)$. To calculate the Frobenius 
trace at $x$ it is convenient to base-extend all our schemes, originally over 
$R_N$ to the finite field $\mathbf{F}_q = R_N/\mathfrak{p}$. Then each element $\lambda$ of the set
$U(\mathbf{F}_{q^d})$ is a vector $(\lambda _0, ..., \lambda _{r+1}) \in \mathbf{F}_{q^d}$, and 
we define $h(x, \lambda ) \in \mathbf{F}_{q^d}[x]$ by specializing $h(x)$ in section  \ref{SS:cyclo2} to 
these values of $\lambda$. The fiber
 \[
 u^{-1}(\lambda) = \mathbb{A}_x^{1}-\{ h(x, \lambda )=0 \}. 
 \]
 We define 
$\bar{\lambda} $ as the geometric point 
$\bar{\lambda}: \mr{Spec} (\bar{\mbf{F}}_q)\to  \mr{Spec} (\mbf{F}_{q^d})\overset{\lambda}{\longrightarrow}U$. 
That is, we regard each of our finite fields as contained in a fixed algebraic closure $\bar{\mbf{F}}_q$.

Choose a primitive character $\chi : \mu _N \to \bar{\QQ}_{\ell} ^{\times}$. For each finite field 
$\mathbf{F}_q$ such that $q \equiv 1$ mod $N$  we define the character
\[
\chi _q : \mathbf{F}_q ^{\times} \to \bar{\QQ}_{\ell} ^{\times}, \quad t \mapsto \chi (t^{(q-1)/N}), \quad
\]
which is meaningful since $t^{(q-1)/N} \in \mu_N$.


\begin{Theorem}\label{T:main}
Let $K_N = \QQ (\zeta _N)$. Assume that
$N$ does not divide any $i_j$ or $i_0+...+i_{r+1}$. 
Then there is a lisse $\bar{\QQ}_{\ell}$-adic sheaf $\mathcal{P}[\mathbf{i}/N, \chi]_{\ell}$ on $U[1/\ell]$ corresponding to a 
representation 
\[
\sigma_{\ell} : \pi _1 (U, \bar{\xi}) \to GL_{r+1}(\bar{\QQ}_\ell)
\]
whose Frobenius traces for $\lambda \in U(\mathbf{F}_{q^d})$ are given by 
\[
\mr{Tr}(\mr{Frob} _{\lambda}   \mid \mathcal{P}[\mathbf{i}/N, \chi]_{\ell, \bar{\lambda}})=
-\mathbb{P}[\mathbf{i}/N, \chi ; \lambda; q^{\deg (\lambda)}] :=  
-\sum _{\stackrel{x\in  \mathbf{F} _{q^d}} {h(x, \lambda) \neq 0}}  \chi _{q^{\deg (\lambda)}}(f_{\mathbf{i}}(x)).
\]
Here 
$\bar{\xi} = \mr{Spec}{\overline{K_N (\lambda)}}$ is a geometric generic point. This sheaf is punctually 
pure of weight 1. 


\end{Theorem}

\begin{proof}

We calculate the Frobenius trace
$\mr{Frob}_{\lambda}$ in the geometric fiber at $\bar{\lambda}$ in the $\ell$-adic realization of the motive
\[
\mathcal{P}[\mathbf{i}/N, \chi]:= 
Ru _{!} f_{\mbf{i}} ^* K(\chi).
\]
By the Grothendieck-Lefschetz
formula, the trace $\mr{Frob}_{\lambda}$ on $R u _{!}\,f _{\mbf{i}}^{\ast} K(\chi) _{\ell}$, i.e., 
the alternating sum
\[
\sum _{i = 0} ^2 (-1)^i \mr{Tr}(\mr{Frob}_{\lambda} \mid   R^i u _{!}\,f _{\mbf{i}}^{\ast} K(\chi) _{\ell, \bar{\lambda}} ) = 
\sum _{i = 0} ^2 (-1)^i \mr{Tr}(\mr{Frob}_{\lambda} \mid   H^i  _c( u ^{-1}(\bar{\lambda}), f _{\mbf{i}}^{\ast} K(\chi) _{\ell} )),
\]
is the sum
\[
\sum _{x \in u ^{-1}(\lambda)(\mbf{F}_{q^d})}  \mr{Tr}(\mr{Frob}_{x}  \mid  f _{\mbf{i}}^{\ast} K(\chi) _{\ell, \bar{x}}  ). 
\]
The local traces are $\chi _{q^{d}}(f_{\mbf{i}}(x))$, where $d = \deg(\lambda)$.  On the other hand we have shown that 
we have $ H^i  _c = 0$ for $i \neq 1$, and that the dimension of $ H^1  _c$ is $r+1$. 
Therefore the local system $R^1 u _{!}\,f _{\mbf{i}}^{\ast} K(\chi) _{\ell}$ gives us a representation
$\sigma_{\ell} : \pi _1 (U, \bar{\xi}) \to GL_{r+1}(\bar{\QQ}_\ell)$
and  
\[
-\mr{Tr}(\mr{Frob}_{\lambda} \mid   R^1 u _{!}\,f _{\mbf{i}}^{\ast} K(\chi) _{\ell, \bar{\lambda}} ) = 
\sum _{\stackrel{x\in  \mathbf{F} _{q^d}} {h(x, \lambda) \neq 0}}  \chi _{q^{\deg (\lambda)}}(f_{\mathbf{i}}(x))
\]
as claimed. That it is punctually pure of weight 1 follows from the fact that 
\[
 H^1  _c( u ^{-1}(\bar{\lambda}), f _{\mbf{i}}^{\ast} K(\chi) _{\ell} )=
H^1 _c (X ^{\circ}_{\bar{\lambda}}, \bar{\QQ}_{\ell}) ^{\chi} = H^1  (X _{\bar{\lambda}}, \bar{\QQ}_{\ell}) ^{\chi}
\]
for any primitive character. On the right-hand side is the cohomology of a projective, nonsingular curve, so this is pure of weight 1.

\end{proof}

It is worth noting that the classes of Frobenius elements of the closed points in $U$ are dense in 
$\pi _1 (U, \bar{\xi})$. Because the representation on $H^1$ of the smooth projective curve preserves the symplectic cup-product 
up to similitude, the representation preserves a Hermitian form up to similitude. For instance, when $N=3$,
$r=2$, $i_0=...= i_3=1$, we get the Picard family of curves. The monodromy of this family is in the group 
$\mr{SU}(2, 1)$.  See section \ref{S:picard}.

\section{Rigid local systems}
\label{S:rigid}

\subsection{\rm{}}
\label{SS:rigid1}

We give a brief survey of the main results of \cite{Katz96}. We consider local systems 
$\mc{F}$ on an open subset $U = \A ^1 - S = \PP^1 - (S \cup \infty)$ where $S$ is a finite set of points. This 
is studied in two contexts: 
\begin{itemize}
\item[1.] The $\mc{D}$-module setting. Then $\mc{F} = {\sf  V} $ is a local system of
$\C$-vector spaces on $U^{\text{an}}$ which is the solution sheaf to a (integrable) algebraic 
differential equation
\[
\nabla : \mc{V} \to \Omega^ 1 _{U/\C} \otimes _{\mc{O}_X} \mc{V}
\]
with regular singular points at $S$. This is equivalent to the monodromy representation
($n  =  \#S$)
\[
\rho : \pi _1 (U^{\text{an}},  x) \simeq
\langle 
\gamma _1, ..., \gamma _n , \gamma _{\infty} \mid \gamma _1 ...,\gamma _n \gamma _{\infty} = 1\rangle
\to  \mr{GL} ( {\sf V}_x)  \sim \mr{GL}_n(\C). 
\]
\item[2.] The $\ell$-adic setting.  Then $\mc{F} $ is a lisse $\bar{\QQ}_{\ell}$-sheaf on $U$. Here 
the ground field $k$ is algebraically closed. The sheaf  $\mc{F} $ is equivalent to a continuous representation
\[
\rho : \pi _1 (U,  x) \to  \mr{GL} ( \mc{ F}_x)  \sim \mr{GL}_n(  \bar{\QQ}_{\ell}   )
\]
where the $\pi _1$ refers to the profinite fundamental group in a geometric point $x$.
\end{itemize}
Let  $X$ be a projective smooth connected curve over $k$ and $x \in X$ is a point. By a {\it disk} at $x$ we mean either a subset 
$D_{(x)}\subset X $ containing $x$ and  homeomorphic to  a disk in the complex plane (in case 1 above), or the strict henselization 
$D_{(x)}\:= \mr{Spec} (\mc{O} ^h _{X, x})$ of the 
local ring at $x$ (case 2). In both cases, we let $D^{*}_{(x)}$ be the punctured disk, i.e., 
$D_{(x)} - \{ x\}$. Given a local system $\mc{F}$ on $X- S$ where $S = \{s_1, ..., s_m \}$ is a finite set of points, 
we get by restriction local systems $\mc{F} {(s)}$ on $D^* _{(s)}$, $s \in S$.  We regard these 
$\mc{F} {(s)}$ up to isomorphism. In case 1, the datum $\mc{F} {(s)}$  is equivalent 
to a representation of $\pi _1 (D^* _{(s)} ) \sim \ZZ$, hence to a matrix $T_s$, well-defined up to conjugacy.
In case 2 we obtain a representation of the inertia group 
\[
\rho (s) : I (s) \to \mr{GL}_n(  \bar{\QQ}_{\ell}   )
\]
well defined up to conjugation. Especially important are those representations that 
factor through the tame inertia $I(s) \to I(s) ^{\mr{tame}} = 
\widehat{\ZZ}(1)_{\mr {not\ } p} = \varprojlim _N\mu _N(\overline{k (s)})$. 

\begin{definition}
We say that a local system $\mc{F}$ on $U = X - S$ is rigid if given any other 
local system  $\mc{G}$ on $U = X - S$ such that for all $s \in S$ there are isomorphisms
$\mc{F} {(s)} \simeq \mc{G} {(s)}$ on $D^* _{(s)}$, $\mc{F}$ is isomorphic to 
$\mc{G}$ as local systems on $U$.
\end{definition}
We only consider this notion in the case of genus zero, i.e., $X = \PP^1$. 
In the analytic case this has the concrete interpretation as follows. The local system 
$\mc{F} $ is equivalent to giving complex matrices $M_1, ..., M_m$ such that 
$M_1 ...M_m = 1$, and similarly $\mc{G} $ is equivalent to giving complex matrices $N_1, ..., N_m$ such that 
$N_1 ...N_m = 1$. That they are locally isomorphic at each $s$ means that there are invertible matrices
$A_i$ such that $N_i = A_i M_i A_i ^{-1}$ for $i=1, ..., m$. Rigidity means that there is a single invertible matrix
$B$ such that $N_i = B M_i B^{-1}$  for $i=1, ..., m$.

\subsection{\rm{}}
\label{SS:rigid2}
A related notion of cohomological rigidity is introduced. Let $j : U \to \A^1$ be the inclusion 
of a nonempty open subset (schemes over $k= \bar{k}$). Let $h : \A ^1\to  \PP^1$ be the inclusion. Let 
$\mc{F}$ be an irreducible lisse $\qlbar$-sheaf on $U$. Then $K = j_{\ast}\mc{F}[1]$, as an element of 
$D^b _c (\A^1, \qlbar)$, is an irreducible perverse and nonpunctual sheaf. For such a sheaf 
we have the index of rigidity
\[
\mr{rig}(\mc{F}) : = \chi (\PP^1, h_{\ast}j_{\ast}\underline{\mr{End}}(\mc{F})). 
\]
We say that such a sheaf is cohomologically rigid if $\mr{rig}(\mc{F}) =2$.
\begin{theorem}
\label{T:rigid1}
 \cite[Thm. 5.0.2]{Katz96}.
Let $\mc{F}$ be a cohomologically rigid local system on an open set $U$ as above.
If $\mc{G}$ is another  lisse $\qlbar$-sheaf on $U$ which is locally isomorphic 
to $\mc{F}$ at all points $s$ of $\PP^1 - U$ in the sense that the representations 
of inertia $I(s) $ given by $\mc{F}(s)$ and $\mc{G}(s)$ are isomorphic, we have an isomorphism 
of lisse $\qlbar$-sheaves on $U$, $\mc{F} \cong \mc{G}$.
\end{theorem}
Note that this theorem requires an 
algebraically closed ground field $k$. If one starts from $\mc{F} =\mc{F}_0\otimes _k \bar{k}$, 
and  $\mc{G} =\mc{G}_0\otimes _k \bar{k}$ for lisse $\mc{F}_0, \mc{G}_0$ defined over a 
nonalgebraically closed field $k$, then the rigidity conclusion is that 
$\mc{G}_0 \cong \mc{F}_0\otimes \Phi$ for a $\qlbar$-valued character 
$\Phi $ of $\mr{Gal}(\bar{k}/k)$ (assuming $\mc{F}$ and $\mc{G}$ are irreducible).

Define the category $\mc{T}_{\ell}$ as the full subcategory of constructible  $\qlbar$-sheaves
$\mc{F}$
 on $\A^1$, such that 
\begin{itemize}
\item[1.] $\mc{F}$ is middle extension: there exists a dense open set 
$j : U \to \A^1$ such that $j^* \mc{F}$ is lisse and irreducible and
$j_{\ast}j^* \mc{F}\cong \mc{F}$. 
\item[2.] $\mc{F}$ is tame: $j^* \mc{F}$ is tamely ramified at each point
of $\PP^1 -U$.
\item[3.] $\mc{F}$ has at least two finite singularities: there are at  least two 
distinct points of $\A^1$ where $\mc{F}$ fails to be lisse.
\end{itemize}

One of the main results of \cite{Katz96} is a classification of objects 
of $\mc{T}_{\ell}$ which are are
\begin{itemize}
\item[1.]
 lisse on $\A^1 - \{ \alpha _1, ..., \alpha _n\}$ ($n \ge 2$; $\alpha _1, ...,\alpha _n$ 
 fixed geometric points). 
\item[2.] cohomologically rigid, and
\item[3.] all eigenvalues of all local monodromies are $N$th roots of unity. Here
$N$ is an integer invertible in $k$.
\end{itemize}

\subsection{\rm{}}
\label{SS:rigid3}
The main results are Theorems 5.2.1, 5.5.4, and 8.4.1 of \cite{Katz96}. First it is shown that 
every such  object  is obtained, starting from objects of (generic) rank one 
in  $\mc{T}_{\ell}$,  by repeated iteration of two constructions 
\begin{itemize}
\item[1.] $\mc{F} \mapsto \rm{MT} _{\mc{L}}(\mc{F})$ (middle tensor product), where
$\mc{L}$ is a rank one object in  $\mc{T}_{\ell}$.
\item[2.] $\mc{F} \mapsto \rm{MC} _{\chi}(\mc{F})$ (middle convolution), where
$\chi : \pi _1 ^{\mr{tame}} (\mathbb{G}_m/k)\to \mu _N (\qlbar)$ is a nontrivial character, 
with corresponding Kummer sheaf $\mc{L}_{\chi}$.
\end{itemize}
The effect of these operations on the local monodromies is determined in 3.3.6 and 
3.3.7 of loc. cit. The rank one objects are tensor products of translated Kummer sheaves
\[
\bigotimes _i \mc{L}_{\chi _i (x - \alpha _i)}.
\]
In Theorem 8.4.1 these rigid local systems are given a motivic interpretation. The motives are eigenspaces
of the cohomology of certain hypersurfaces. Define
\[
R_{N, \ell}:= \ZZ[\zeta _N, 1/N \ell]
\]
and
\[
S_{N, n,\ell} := R_{N, \ell}[T_1, ..., T_n][1/\Delta], 
\quad \Delta = \prod_{ i \neq j}(T_i - T_j).
\]
One fixes an embedding $R_{N, \ell}\to \qlbar$ (equivalently, a primitive 
$N$th root of unity in $\qlbar$). For each $r \ge 0$ define
\[
\A(n, r+1)_{R_{N, \ell}}= 
\mr{Spec} (R_{N, \ell}[T_1, ..., T_n, X_1, ..., X_{r+1}][1/\Delta_ {n, r}])
\]
where
\[
\Delta_ {n, r} = \prod_{ i \neq j}(T_i - T_j)
\prod_{a, j}(X_a -T_j)\prod_k (X_{k+1}- X_k)
\]
($i, j \in \{ 1, ..., n\}$, $a\in \{ 1, ..., r+1\}$, $k \in \{1, ..., r \}$). When $r =0$ the 
last factor is the empty product, interpreted as $1$. In $\mathbb{G}_m \times \A(n, r+1)_{R_{N, \ell}}$
consider the hypersurface $\mr{Hyp}(e, f)$ with equation
\[
Y^N = \left(\prod_{a, i} (X_a - T_i)^{e(a, i)} \right) 
\left(    \prod _{k = 1}^r    (X_{k+1}-X_k)^{f(k)} \right)
\]
where the integers $e(a, i)$ are arbitrary, and none of the integers
$f(k)$ is divisible by $N$. Let 
\[
\pi: \mr{Hyp}(e, f)\to (\A^1 - \{ T_1, ..., T_n\})_{S_{N, n,\ell}}
\]
be the map 
\[
(Y, T_1, ..., T_n, X_1, ..., X_{r+1}) \mapsto (T_1, ..., T_n, X_{r+1}).
\]
In other words, we regard $X_{r+1}$ as the coordinate on the target
$\A^1$, and we think of $\mr{Hyp}(e, f)$ as a family of hypersurfaces 
in $(Y, X_1, ..., X_r)$-space parametrized by $(T_1, ..., T_n, X_{r+1})$.
In this setup, the parameter $X_{r+1}= \lambda$ is distinguished; we 
get local systems on the $\lambda$-line minus the points 
with coordinates $\lambda = T_1$, ..., $\lambda = T_n$.

Fix one faithful character $\chi : \mu _N({R_{N, \ell}})\to \qlbar^{\times}$. 
These roots of unity act on $\mr{Hyp}(e, f)$ in the obvious way:
$Y \mapsto \zeta Y$. Katz proves that the sheaves $R^i \pi _{!} \qlbar$ on 
$(\A^1 - \{ T_1, ..., T_n\})_{S_{N, n,\ell}}$ are lisse and tame.
 The eigenspace $(R^i \pi _{!} \qlbar) ^{\chi} $ is 
 nonvanishing
only when $i = r$, and mixed in integral weights in the interval 
$[0, r]$. The  weight $r$ quotient of that sheaf, denoted 
$\mc{H}_{=r}$, if nonzero, when restricted to every geometric
fiber of $(\A^1 - \{ T_1, ..., T_n\})_{S_{N, n,\ell}}$ over
$S_{N, n,\ell}$ is geometrically irreducible and cohomologically rigid, 
all of whose local monodromies have eigenvalues that are $N$th roots of unity. 

He shows 
also that every such rigid local system arises this way. Also he notes that because 
these rigid local systems belong to universal families on the open subset $V_{N, \ell}$ of affine space 
of dimension $n+1$, $\mr{Spec} (R_{N, \ell} [T_1, ...,  T_{n+1}] )$, where
\[
\Delta = \prod _{i \neq j} (T_i -T_j) \neq 0,
\]
the representation of $\pi _1 (\A^1 - \{\alpha_1, ..., \alpha _n \}, \bar{x})$ afforded by any rigid local system 
in $\mc{T}_{\ell}$ extends to a representation of  $\pi _1 (V_{N, \ell}, \bar{x})$. Geometrically, this fundamental 
group is Artin's braid group on $n+1$ letters.

\section{Appell-Lauricella systems}
\label{S:AL}
\subsection{\rm{}}
\label{SS:AL1}
This is the case $r=1$ of the above constructions. We get a family of curves 
in the $(X_1, Y)$-plane 
\[
Y^N = \prod_{ i} (X_1 - T_i)^{e(1, i)} \prod_{ i} (X_2 - T_i)^{e(2, i)} 
     (X_2-X_1)^{f(1)} .
\]
The factor  $\prod_{ i} (X_2 - T_i)^{e(2, i)}$ comes from the base, 
i.e., $\mr{Spec}(S_{N, n,\ell})$ and its effect on 
$\mc{H}_{r = 1}$ is a twist by the Kummer sheaf

\[
\bigotimes _i\mc{L} _{\chi_{2, i} (X_2 - T_i)}.
\]
This can be omitted, and so we are considering the family of curves
in the $(X_1, Y)$-plane ($\lambda = X_2$)
\[
Y^N = \prod_{ i = 1}^n (X_1 - T_i)^{e_i} 
     (\lambda-X_1)^{f}.
\]
We let $(T_1 = \alpha_1, ..., T_n= \alpha _n)$ take on fixed values in an algebraically closed field, and
we get a one-parameter family of curves $C_{\lambda}$. The stalk of 
$\mc{H}_{r=1}$ at a geometric point $\lambda$ is then
\[
(H^1 _c (C_{\lambda}, \qlbar) ^{\chi}) _{r=1} = H^1 (\tilde{C}_{\lambda}, \qlbar) ^{\chi}.
\]
where $\tilde{C}_{\lambda}$ is the projective nonsingular model of $C_{\lambda}$. 
We can calculate this as follows. Consider the projection
$\rho: \tilde{C}_{\lambda} \to \PP^1 $ where the target has affine coordinate 
$X_1$. This is a Galois  $\mu_N $-covering over $U = \PP^1 - \{ \alpha_1, ..., \alpha _n, \lambda, \infty \}$. 
$\rho _{\ast}\qlbar$ is a constructible sheaf on $\PP^1$, lisse of rank $N$ on $U$. 
We have a decomposition into eigenspaces ($\chi$ is a primitive character
of $\mu_N$)
\[
\rho _{\ast}\qlbar = \bigoplus _{j = 0} ^{N}\rho _{\ast}\qlbar ^{\chi ^j}.
\]
Then $\rho _{\ast}\qlbar ^{\chi ^0} = \qlbar $, and for any 
$\phi \neq 1$, over the open set $U$, we get the Kummer sheaf
\[
\rho _{\ast}\qlbar ^{\phi}\mid U  = \mc{L} _{\phi (\prod_{ i = 1}^n (X_1 - \alpha_i)^{e_i}  (\lambda-X_1)^{f})}\mid U. 
\]
The behavior at the ramification points $\{ \alpha_1, ..., \alpha _n, \lambda, \infty \}$ depends on the nature
of the integers $e_1, ..., e_n, f$ modulo $N$. The simplest choice is to assume 
that each $e_i$, $f$ and $e_1 + ...+e_n +f $ are relatively prime to $N$. Then 
\[
\rho _{\ast}\qlbar ^{\phi} = j_{\ast}  \mc{L} _{\phi (\prod_{ i = 1}^n (X_1 - \alpha_i)^{e_i}  (\lambda-X_1)^{f})}
 \]
where $j : U \to \PP^1$ is the inclusion.
\subsection{\rm{}}
\label{SS:AL2}
Curves with equations $Y^N = \prod_{ i = 1}^n (X_1 - \alpha_i)^{e_i} (\lambda-X_1)^{f}$ are sometimes called cycloelliptic. 
It is more natural to think of them as depending simultaneously on the parameters
$\{\alpha _1, ..., \alpha_n, \lambda \}$. As such, they define local systems on the space of parameters $\{\alpha _1, ..., \alpha_n ,\lambda \}$ minus the hyperplanes where two of these coordinates agree. Deligne and Mostow \cite{DM} made an extensive study of the corresponding monodromy groups. Also, the structure of the $\mc{D}$-modules for these families is worked out 
by Holzapfel in part 2 of the book \cite{Holz1} (see also \cite{Holz2}). To our knowledge, the $\ell$-adic (or $p$-adic)
story of these is only partially available. Lei  Fu studies $\ell$-adic analogs 
of GKZ hypergeometric systems in  \cite{LF}. Katz's theory is for local systems in one variable, i.e., 
on the line. This explains the singling out of the distinguished parameter $\lambda$.

\section{Examples: Rigidity}
\label{S:exrig}

\subsection{\rm{}}
\label{SS:ex1} Recall that the group of automorphisms of 
$\PP^1$ that permute the set $0, 1, \infty$ has order 6 and is generated by 
$\mr{inv}(x) =  1/x$ and $g(x) = 1-x$. 

\begin{lemma}
\label{L:ex1} If $a=b$, 
\[
\mr{inv}^* \mc{H}^{\mr{can}} (\alpha _i; \beta _j) = \mc{H}^{\mr{can}} (\bar{\beta} _j; \bar{\alpha} _i) .
\]
\end{lemma}
\begin{proof}
We have 
\[
\mc{H}^{\mr{can}} (\alpha _i; \beta _j)[1]  = \mc{H}^{\mr{can}} (\alpha _1; \beta _1) [1]
*_{!}...*_{!}
\mc{H}^{\mr{can}} (\alpha _a; \beta _a)[1]. 
\]
Since $\mr{inv}$ is an automorphism of the algebraic group $\mbf{G}_m$, it commutes 
with this convolution, so it suffices to prove the lemma when $a=1$. The equality 
clearly holds over $U = \mbf{G}_m - \{ 1 \}$ geometrically (i.e., over an algebraically closed field)
since both sides have the same monodromy at $0, 1, \infty$. By rigidity, we get
\[
\mr{inv}^* \mc{H}^{\mr{can}} (\alpha _1 \beta _1) = \mc{H}^{\mr{can}} (\bar{\beta} _1; \bar{\alpha} _1) 
\otimes  C^{\mr{deg}}
\]
for a constant $C \in \qlbar$. We see that $C=1$ by comparing the trace of Frobenius of both 
sides and using the elementary identity $J(\mu, \nu ) = \nu (-1)J(\bar{\mu}\bar{\nu}, \nu)$
(see equation \eqref{E:charshv1} in section  \eqref{SS:charshv4}). Since the equality 
holds over $U$ it holds over all $\PP^1$ by applying $j_{\ast}$ for $j: U \to \PP^1$.

\end{proof}

The situation for $g(x) = 1-x$ is more complicated. This is not an automorphism 
of the group $\mbf{G}_m$, and it does not commute with convolution. 

\begin{lemma}
\label{L:2}  Let $g(x) = 1-x$. Then over  the open set
$U = \mbf{G}_m - \{1\}$ there is an isomorphism
\[
g^*\mc{H}^{\mr{can}} ( \alpha _1, \alpha _2 ;  \beta _1, \beta _2) \otimes
\mc{L}_{\bar{\alpha} _2 (x-1)} \cong 
\mc{H}^{\mr{can}} (\beta _1 \beta _2 \bar{\alpha} _1\bar{ \alpha} _2 , 1;  \beta _1\bar{\alpha} _2, \beta _2\bar{\alpha} _2) 
\otimes C^{\mr{deg}}
\]
for an explicitly computable $C \in \qlbar$.
\end{lemma}
\begin{proof}
This argument assumes semisimple monodromy. 
Both sides have monodromy 
\[
\begin{pmatrix}
 \beta _1 \beta _2 \bar{\alpha} _1\bar{ \alpha} _2& 0\\
0 &1
\end{pmatrix}, \ \ \ 
\begin{pmatrix}
1 & 0\\
0 &\alpha _1 \bar{\alpha} _2 
\end{pmatrix}, \ \ \ 
\begin{pmatrix}
\bar{\beta} _1 \alpha _2 & 0\\
0 &\bar{\beta} _2 \alpha_2
\end{pmatrix}
\]
at $0, 1, \infty$ respectively. By rigidity, we have an isomorphism as above, with the constant 
$C$ to be computed. This is done by evaluating the trace of Frobenius of both sides at $1/2$. 
By lemma  \eqref{L:charshv1} the traces of Frobenius at $1/2$ for  
$g^*\mc{H}^{\mr{can}} ( \alpha _1, \alpha _2 ;  \beta _1, \beta _2) \otimes
\mc{L}_{\bar{\alpha} _2 (x-1)}$ and $\mc{H}^{\mr{can}} (\beta _1 \beta _2 \bar{\alpha} _1\bar{ \alpha} _2 , 1;  \beta _1\bar{\alpha} _2, \beta _2\bar{\alpha} _2) $ are respectively 
\begin{align*}
&A\alpha _2 (1/2) \bar{\alpha} _2 (-1/2) \sum _{ t_1} \lambda(t_1) \mu(1-t_1) 
\nu(1/2-t_1),\quad 
\lambda = \alpha _1 \beta _2 ^{-1}, \mu = \beta _1 \alpha _1 ^{-1},
\nu = \beta _2 \alpha _2 ^{-1}\\
&B  \sum _{ u_1} \mu(u_1) \lambda(1-u_1) 
\nu(1/2-u_1)
\end{align*}
where $A, B$ are computed constants involving Jacobi sums
(lemma  \eqref{L:charshv1}). Replacing $t_1 \mapsto -t_1$
in the first formula gives
\[
A\alpha_1 (-1)  \sum _{ t_1} \lambda(t_1) \mu(t_1+1) 
\nu(-1/2-t_1), 
\]
then replacing $t_1 \mapsto u_1 -1$ gives $C = A\beta _2  (-1) /B$  times the second sum.
\end{proof}

 If $\alpha _2\neq 1$, we cannot assert an isomorphism over all
of $ \mbf{G}_m $. The right-hand side has a one-dimensional stalk at $1$ whereas 
$\mc{L} _{\bar{\alpha}_2 (x-1)}$ has a zero-dimensional stalk there. 
The map $g$ exchanges the singular points $0, 1$, but we cannot 
use it to compute the Frobenius action on 
$\mc{H}^{\mr{can}} ( \alpha _1, \alpha _2 ;  \beta _1, \beta _2) _0$
from the action of Frobenius on 
 $\mc{H}^{\mr{can}} (\beta _1 \beta _2 \bar{\alpha} _1\bar{ \alpha} _2 , 1;  \beta _1\bar{\alpha} _1, \beta _2\bar{\alpha} _2)_1$. Recall that we view these sheaves as living on 
 $\PP^1$, extending via $j_{\ast}$ where $j : U \to \PP^1$.
 However, it will work if $\alpha _2 =1$. 
 
 \begin{corollary}
 \label{C:ex1} There is an isomorphism on all of $\PP^1$:
 \[
g^*\mc{H}^{\mr{can}} ( \alpha _1, 1 ;  \beta _1, \beta _2)  \cong 
\mc{H}^{\mr{can}} (\beta _1 \beta _2 \bar{\alpha} _1 , 1;  \beta _1, \beta _2) 
\otimes C^{\mr{deg}}
\]
for an explicitly computable $C \in \qlbar$.
 \end{corollary}
\begin{proof}
The map $g$ is an automorphism of $U$ and it extends to an automorphism 
of $\PP^1$. Over $U$ we have proved the isomorphism, so we can apply 
$j_{\ast}$ to the equation. Note that this argument would fail if we had the 
additional factor $\mc{L}_{\bar{\alpha} _2 (x-1)} $ because 
$j_{\ast}$ does not commute with tensor product in general ($\otimes C^{\mr{deg}}$ is not a problem).
 
\end{proof}

\subsection{\rm{}}
\label{SS:exrig2}  
Here is an example of a quadratic transformation.
Let $\varepsilon$ be the Legendre character of $\FF_q ^{\times}$ ($q$ is odd; 
$\varepsilon ^2 =1, \varepsilon \neq 1$). We let
$\beta _1, \beta _2$ be characters of $\FF_q ^{\times}$ such that none of 
$\beta _1 ^2, \beta _2^2, \beta _1\beta _2 \varepsilon$ is $1$. 

\begin{proposition}
\label{P:1} Let $\mc{H} := \mc{H} ^{\mr{can}}(\varepsilon, 1; \beta _1, \beta _2)$, and 
 $\mc{K} := \mc{H} ^{\mr{can}}(\beta _1 \beta _2\varepsilon, 1; \beta _1^2 , \beta _2^2)$. Then we have an equality
 of trace functions  $t_{\mc{H}} (x^2)= C t_{\mc{K}}((x+1)/2)$ for an explicitly computable 
$C \in \qlbar$. 
 \end{proposition}

\begin{proof}
The sheaf 
$\mc{H} := \mc{H} ^{\mr{can}}(\varepsilon, 1; \beta _1, \beta _2)$ has monodromy
\[
\begin{pmatrix}
\varepsilon & 0\\
0 &1
\end{pmatrix}, \ \ \ 
\begin{pmatrix}
1 & 0\\
0 &\beta _1 \beta _2 \varepsilon
\end{pmatrix}, \ \ \ 
\begin{pmatrix}
\beta _1 ^{-1} & 0\\
0 &\beta _2 ^{-1}
\end{pmatrix}
\]
respectively at $0, 1, \infty$. Let $[2]:  \mbf{G}_m \to \mbf{G}_m$ be the map 
$t \mapsto t^2$. Then $[2]^{\ast }\mc{H} $ is lisse on $\mbf{G}_m$ except possibly
at $0, -1, 1, \infty$, where the monodromies are respectively 
\[
\begin{pmatrix}
\varepsilon & 0\\
0 &1
\end{pmatrix}^2, \ \ \ 
\begin{pmatrix}
1 & 0\\
0 &\beta _1 \beta _2 \varepsilon
\end{pmatrix}, \ \ \ 
\begin{pmatrix}
1 & 0\\
0 &\beta _1 \beta _2 \varepsilon
\end{pmatrix}, \ \ \ 
\begin{pmatrix}
\beta _1 ^{-1} & 0\\
0 &\beta _2 ^{-1}
\end{pmatrix}^2.
\]
Since $\varepsilon ^2 =1$, the first one is the identity matrix, and thus
$[2]^{\ast }\mc{H} $ is lisse at $0$. Let
 $h:  \mbf{G}_m \to \mbf{G}_m$ be the map $h(t) = (t+1)/2$. Then 
 $h^*\mc{K}$ has monodromy at $-1, 1, \infty$ given by the last three matrices above.
By rigidity, we have a geometric isomorphism  $h^*\mc{K} \cong [2]^{\ast }\mc{H}$, and 
since these are irreducible, they are isomorphic up to a twist 
$h^*\mc{K} \cong [2]^{\ast }\mc{H} \otimes \Phi$, for a character
 $\Phi $ of $\mr{Gal}(\bar{\FF}_q/\FF _q)$ whose value on a Frobenius 
 generator is a unit $C$ in $\qlbar$. We can compute this number in several ways. 
One way is to observe that $[2](1) = 1^2 = h(1)$. Since $[2]$ induces an isomorphism
from the henselization of the local ring at $1$ to the  henselization of the local ring at $1$,
the Frobenius actions on the stalks $\mc{H}_1$ and $([2]^*\mc{H})_1$ coincide. 
On the other hand $h$ is an isomorphism, so the Frobenius actions on 
$\mc{K}_1$ and $(h^* \mc{K})_1$ coincide. Therefore,  $C$ will be the ratio 
of the traces of Frobenius on $\mc{H}_1$ and $\mc{K}_1$. In general
(see corollary \ref{C:charshv1}), 
for $\mc{F} =\mc{H}^{\mr{can}} ( \alpha _1, \alpha _2 ;  \beta _1, \beta _2) $
\[
t_{\mc{F}} (1)= 
 \mr{Tr} (\mr{Frob }_1 | \mc{F}_{\bar{1}}) 
 = -\frac{\beta _2 \bar{\alpha}_2(-1) 
 J(\alpha _1 \bar{\beta}_2, \beta _1 \beta_2\bar{\alpha}_1\bar{\alpha}_2)}
 {J(\alpha _1, \beta_1 \bar{\alpha}_1)J(\alpha _ 2, \beta_2 \bar{\alpha}_2)}.
\]
We get an equality of hypergeometric character sums $t_{\mc{H}} (x^2)= C t_{\mc{K}}((x+1)/2)$, 
where $C =t_{\mc{H}} (1)/t_{\mc{K}} (1) $.

\end{proof}

This identity is the analog of one of Kummer's quadratic transformations. In the language 
of Riemann's $P$-function, this is 
\[
P
\begin{bmatrix}
0 & \infty &1 &\\
0 & a & 0 & x^2\\
\frac{1}{2}& b&\frac{1}{2} -a-b &
\end{bmatrix}= 
P
\begin{bmatrix}
0 & \infty &1 &\\
0 & 2a &0 & (x+1)/2\\
\frac{1}{2} -a-b & 2b&\frac{1}{2} -a-b  &
\end{bmatrix}.
\]
Note that this is less precise than the corresponding formula given by Kummer. 
The above is an equality of 2-dimensional spaces of multivalued holomorphic functions
which are the solution spaces for the corresponding hypergeometric differential equations.

In terms of finite field hypergeometric functions in \cite{FLRST}, we have the following result. For a given odd prime $p$, let $A$, $B$ be any  characters of $\FF_p^\times$ with $A^2$, $B^2$, $\eps \ol AB$, $\eps A\ol B$, $\eps AB\neq 1$. For any $z\in \FF_p$ with $z^2\neq 1$, $0$, we have 
$$
  B(4)\pPPq 21 {A& B}{& 
  \eps}{z^2}= \frac{g(\eps \ol B)g(\eps A)}{g(\eps)g(\eps A \ol B)}\pPPq 21 {A^2& B^2}{& 
  \eps AB}{\frac{z+1}2}.
$$
Equivalently, 
$$
  \pFFq 21 {A& B}{& 
  \eps}{z^2}=\frac{J(\eps A, \eps B)}{J(\eps, \eps AB)}  \pFFq 21 {A^2& B^2}{& 
  \eps AB}{\frac{z+1}2}
  .
$$ To derive these identities, we also use the formula
$$
  g(\chi^2)g(\eps)=\chi(4)g(\chi)g(\eps \chi),
$$
for any character $\chi$.

\subsection{\rm{}}
\label{SS:exrig3}  The general pattern of these identities is in the shape
$t_{\mc{H}} (R(x)) = C t_{\mc{K}} (R(x))$ for two (rigid) local systems 
$\mc{H}, \mc{K}$ and rational functions $R(x), S(x)$. Here $C \in \qlbar$, but in 
fact, the constant  $C$ is an algebraic number. The dependence
of $C$ on the various parameters appearing in $\mc{H}, \mc{K}$ is an
interesting problem. In the previous example, the expression 
for $C = C(q; \varepsilon, \beta_1, \beta_2)$ in terms of Jacobi sums
shows that, in an appropriate sense, 
\begin{itemize}
\item[1.] For a fixed prime $p$, $C(p^e; \varepsilon, \beta_1, \beta_2)$ is a 
$p$-adic analytic function of the $\beta _1, \beta _2$. This follows from the 
Gross-Koblitz formula for Gauss sums, \cite{GK}. 
\item[2.] For fixed $\beta _1, \beta _2$, $C(q; \varepsilon, \beta_1, \beta_2)$ defines a Hecke 
character (an automorphic form for $\mr{GL}_1$) of a cyclotomic field, \cite{Weil1}, \cite{Weil2}.
\end{itemize}

\section{Examples: Arithmetic Triangle Groups}
\label{S:ATG}
The basic idea here is based on the following observation: 
Let
$X_1 := \Gamma _1\backslash \mathfrak{H}^*$, $X_2 := \Gamma _2\backslash \mathfrak{H}^*$ be the Riemann surfaces
obtained as  quotients of the suitably compactified complex upper half plane
by  triangle groups $\Gamma _1$, $\Gamma_2$, respectively. Further assuming $\Gamma _1 \subset \Gamma _2 $,  we obtain a covering of corresponding  Riemann surfaces $X_1 \to X_2$.  There is a hypergeometric DE attached 
to a triangle group: the DE belonging to the Schwarz uniformization. 
The Schwarzian differential equations pull back under coverings.
This is slightly complicated by the fact that the Schwarzian DE 
is a third order equation for the ratio $y_1/y_2$ of a second 
order DE which is only well-defined up to a twist. This means 
that extra factors can occur in the formulae for the second 
order equations.

We say a triangle group $\Gamma \subset \slr$ is {\it arithmetic}
if it arises from a quaternion algebra over a totally real number field. 
These have been classified by Takeuchi, see \cite{Tak}. A vast
generalization appears in the work of Deligne and Mostow in \cite{DM}.

Here we give some examples arising from arithmetic triangle groups. 
\begin{example}[A cubic formula from the groups $(2,4,8)$ and $(2,3,8)$.]

We have the cubic transformation between the hypergeometric functions:

 \[
\pFq {2}{1} {\frac {1}{48},  \frac{17}{48}}
{  ,\frac {1}{2}} {\frac{x(x-9)^2}{(x+3)^3}}=   \left (1+\frac x3\right )^{1/16}
  \pFq {2}{1} {\frac {1}{16},  \frac{3}{16}}{ ,\frac {1}{2}}{x}.
  \]

In the language 
of Riemann's $P$-function, this is 
\[
 \left (\frac 3{x+3}\right )^{1/16}  P
\begin{bmatrix}
0 & \infty &1 &\\
0 &  \frac 1{48} &0 & \frac{x(x-9)^2}{(x+3)^3}\\
\frac{1}{2}& \frac{17}{48} & \frac 18&
\end{bmatrix}=
P
\begin{bmatrix}
0 & \infty &1 &\\
0 & \frac 1{16} &0 & x\\
\frac{1}{2}&\frac{3}{16}&\frac 14  &
\end{bmatrix}.
\]

Set $t= \frac{x(x-9)^2}{(x+3)^3}$. Then the values of $x$ for $t=0$, $1$, $\infty$, are as follows:
$$
  \begin{array}{c||c|c|c}
       t& 0 & 1& \infty \\ \hline 
       x& 0, 9, 9& 1, 1, \infty& -3,-3,-3
  \end{array}
$$

In terms of Katz's hypergeometric sheaves and finite field hypergeometric functions, we have the following results. For a given prime $p\equiv 1 \mod 48$, let $\eta$ be any primitive character of $\FF_p^\times$ of order $48$. Then we have 
$$
  t_{\mc H}\left( \frac{x(x-9)^2}{(x+3)^3}\right)=  t_{\mc{K}}(x),
$$
where 
 $\mc{H} := \mc{H} ^{\mr{can}}(\varepsilon, 1; \ol\eta, \ol\eta^{17}) $, and 
 $\mc{K} := \mc K^{\mr{can}}( \eps, 1; \ol\eta , \ol\eta^9)\otimes (\mc{K} ^{\mr{can}} (1; \eta^3)\otimes \mc L_{\eta^3, 1-x/3})$. 
Let 
$
 f(z)=z(z-9)^2/(z+3)^3
$. For any $z\in \FF_p$ with $f(z)\neq 0$, $1$, and $\infty$, we have 
$$
  \pFFq 21 {\eta& \eta^{17}}{& 
  \eps}{f(z)}= \eta^3\left (1+z/3\right )\pFFq 21 {\eta^3& \eta^{9}}{& 
  \eps}{z}
  .
$$

The sheaf $\mc H$ has monodromies 
\[
\begin{pmatrix}
\varepsilon & 0\\
0 &1
\end{pmatrix}, \ \ \ 
\begin{pmatrix}
1 & 0\\
0 &\eta^6
\end{pmatrix}, \ \ \ 
\begin{pmatrix}
\eta & 0\\
0 &\eta^{17}
\end{pmatrix}
\]
at $t=0$, $1$ and $\infty$, respectively.  Let $g(x)=x(x-9)^2/(x+3)^3$ . Then $g^\ast \mc H$ has monodromies 
as follows:
$$
  \begin{array}{c||c|c|c|c|c}
       x& 0 & 9 & 1 & \infty& -3 \\ \hline 
       & \begin{pmatrix}
\varepsilon & 0\\
0 &1
\end{pmatrix} 
& 
\begin{pmatrix}
\varepsilon & 0\\
0 &1
\end{pmatrix}^2
  
  & 
  \begin{pmatrix}
1 & 0\\
0 &\eta^6
\end{pmatrix}^2&
\begin{pmatrix}
1 & 0\\
0 &\eta^6
\end{pmatrix}&
\begin{pmatrix}
\eta & 0\\
0 &\eta^{17}
\end{pmatrix}^3
\end{array}
$$
Therefore, 
$g^\ast \mc H \otimes (\mc{H} ^{\mr{can}} (1; \ol\eta^3)\otimes \mc L_{\ol\eta^3,x/(x+3)}) $ and $ \mc K^{\mr{can}}( \eps, 1; \ol\eta , \ol\eta^9)$ have the  monodromies
$$\begin{pmatrix}
\varepsilon & 0\\
0 &1
\end{pmatrix}, \quad  \begin{pmatrix}
1 & 0\\
0 &\eta^{12}
\end{pmatrix}, \quad  \begin{pmatrix}
\eta^3 & 0\\
0 &\eta^9
\end{pmatrix}
$$
at $0$, $1$, and $\infty$, respectively.  This give us the identity between the traces of 
Frobenius and thus the finite hypergeometric functions. 

\end{example}

\medskip
\begin{example}\cite[Entry (116)]{Goursat}
Goursat showed the following cubic transformation of hypergeometric functions
  \[
\pFq {2}{1} {a,  a+\frac {1}{3}}
{, 2a+\frac {5}{6}}{27\frac{x(1-x)^2}{(1+3x)^3}}=   \left (1+ 3x\right )^{3a}
\pFq {2}{1} {3a,  3a+\frac{1}{2}}{,  2a+\frac {5}{6}}{x}.
  \]
 When $a=\frac{2n-1}{24n}$ for a positive integer $n$, the function
$f(x)$ gives the covering map from the curve associated to the arithmetic triangle group $(2,6n,12n)$ to the curve associated to the arithmetic triangle group $(2,3,12n)$.

In the language 
of Riemann's $P$-function, this is 
\[
 \left (\frac 1{1+3x}\right )^{3a}  P
\begin{bmatrix}
0 & \infty &1 &\\
0 &  a &0 &27\frac{x(1-x)^2}{(1+3x)^3}\\
\frac{1}{6}-2a& a+\frac13 & \frac 12&
\end{bmatrix}=
P
\begin{bmatrix}
0 & \infty &1 &\\
0 & 3a &0 & x\\
\frac{1}{6}-2a& 3a+\frac12 & \frac 13-4a&
\end{bmatrix}.
\]

Set $f(x)=27\frac{x(1-x)^2}{(1+3x)^3}$. Then the values of $x$ for $f(x)=0$, $1$, $\infty$, are as follows:
$$
  \begin{array}{c||c|c|c}
       f(x)& 0 & 1& \infty \\ \hline 
       x& 0, 1, 1& 1/9, 1/9, \infty& -1/3,-1/3,-1/3
  \end{array}
$$

In terms of  hypergeometric sheaves and finite field hypergeometric functions, we have the following results. For a given prime $p\equiv 1 \mod 6$, let $\eta$ be any primitive character of $\FF_p^\times$ of order $6$, and $\alpha$ be any character with $\alpha^6\ne 1$. Then we have 
$$
  t_{\mc H}\left( f(x)\right)=  t_{\mc{K}}(x),
$$
where 
 $\mc{H} := \mc{H} ^{\mr{can}}(\ol\alpha^2\ol\eta, 1; \ol\alpha, \ol\alpha\ol\eta^2) $, and 
 $\mc{K} := \mc K^{\mr{can}}(\ol\alpha^2\ol\eta, 1; \ol\alpha^3, \ol\alpha^3\eps)\otimes (\mc{K} ^{\mr{can}} (1; \alpha^3)\otimes \mc L_{\alpha^3, 1-3x})$. 
 For any $z\in \FF_p$ with $f(z)\neq 0$, $1$, and $\infty$, we have 
$$
  \pFFq 21 {\alpha^3& \eps \alpha^3}{&\alpha^2\ol 
  \eta}{z}= \ol \alpha(1+3z)\pFFq 21 {\alpha& \eps \alpha\eta^2}{&\alpha^2\ol 
  \eta}{f(z)}.
  $$

\end{example}
\medskip

\begin{example}\cite[Equation (28)]{Vidunas}
Similarly, we have the finite field version of the following degree-$10$ algebraic transformation:
\begin{align*}
   \left(1-57x-1029 x^2+50421 x^3\right )^{1/28}& 
    \pFq {2}{1} {\frac {5}{84},  \frac{19}{42}}{\phantom{AA} \frac {5}{7}}{27x}  \\
 = &\pFq {2}{1} {\frac {1}{84},  \frac{29}{84}}{\phantom{AA} \frac{6}{7}}{\frac{-27x^2(1-27x)(3-49x)^7}{4(1-57x-1029 x^2+50421 x^3)^3}}
\end{align*}

For a given prime $p\equiv 1 \mod 84$, let $\eta$ be any primitive character of $\FF_p^\times$ of order $84$. Let 
$$
 f(z)=1-57z-1029 z^2+50421 z^3,\quad
 g(z)=-z^2(1-27z)(3-49z)^7.
$$For any $z\in \FF_p$ with $f(z)/g(z)\neq 0$, $1$, and $\infty$, we have 
$$
  \pFFq 21 {\eta^{10}& \eta^{38}}{& 
  \eta^{60}}{27z}= \ol \eta^3\left(f(z)\right)\pFFq 21 {\eta& \eta^{29}}{& 
  \eta^{72}}{\frac {27}4\frac{g(z)}{f(z)^3}}
  .
$$
\end{example}

\section{Examples: Elliptic curves}
\label{S:exelliptic}

(See \cite{St0}, \cite{St1}).
 We consider the differential equations satisfied by the periods of  families
 of elliptic curves. Let 
 \[
 y^2 = 4x^3 -g_2 x-g_3
 \]
 be the Weierstrass family of elliptic curves. $\Delta= g_2 ^3 -27 g_3^2 $ the discriminant. When 
 $D \neq 0$ this is an elliptic curve. One has the differentials of the first and second kind
 \[
 \omega = \frac{dx}{\sqrt{4x^3 -g_2 x-g_3}}= \frac{dx}{y}, \quad 
  \eta = \frac{xdx}{\sqrt{4x^3 -g_2 x-g_3}} = \frac{xdx}{y}.
 \]
 These generate the deRham cohomology of the curve. Recall
\begin{proposition}
(\cite[Proposition 2.5]{Del2}) Above $S = \mr{Spec} (\ZZ [2^{-1}, 3^{-1}]   $ there is a
moduli scheme for  pairs $(E, \omega)$  consisting of  a curve of genus 1 together 
with an invariant invertible differential one form. This scheme is 
\[
\bar{M} = \mr{Spec} (\ZZ [2^{-1}, 3^{-1}] [g_2, g_3]
\]
with universal curve (in nonhomogeneous coordinates)  $y^2 = 4x^3 -g_2 x-g_3$
with invariant differential $\omega = dx/y$. 
\end{proposition} 
In Deligne's formulaire, singular curves are permitted. Precisely, a curve 
of genus 1 over a base $T$ is a proper and flat morphism of finite presentation
$p : E \to T$  together with a section $e$ contained in the open subset of smoothness
of $p$ whose geometric fibers are reduced irreducible curves of arithmetic genus 1. The fibers 
are of three types: 
\begin{itemize}
\item[1.] An elliptic curve, i.e. proper, smooth connected of genus 1;
\item[2.] a projective line in which two distinct points have been identified (cubic 
in $\PP^2$ with an ordinary double point);
\item[3.]  a projective line in which two infinitely near points have been identified (cubic in 
$\PP^2$ with a cusp).

\end{itemize}
 We let $M \subset \bar{M}$ be the open set where $\Delta \neq 0$, and $f:E \to M$ be the 
 universal Weierstrass elliptic curve. $f$ is a smooth morphism.
 We get the deRham cohomology sheaves on $M$, 
 \[
 H_{DR} ^i (E/M) := \mbf{R}^i f_{\ast} \Omega ^{\bullet} _{E/M}. 
 \]
 There is a filtration
 \[
 \begin{CD}
 0 @>>> f_{\ast} \Omega ^{1} _{E/M} @>>>H_{DR} ^1(E/M)@>>> R^1 f_{\ast} \mc{O}_{E}
 @>>>0.
  \end{CD}
 \]
 Locally on $M$,  $H_{DR} ^i (E/M)$ is spanned as an $\mc{O}_M$-module by 
 $\omega, \eta$, with the submodule $ f_{\ast} \Omega ^{1} _{E/M} $ spanned by 
 $\omega$. There is an integrable Gauss-Manin connection 
 \[
 \nabla: H_{DR} ^1 (E/M)  \to \Omega ^1 _{M/S}\otimes _{\mc{O}_M}H_{DR} ^1 (E/M). 
 \]
 which has regular singularities at infinity. Note that $\bar{M}-M$ is not a divisor with normal 
 crossings, but we can compactify $M$ in such a way that the divisor at infinity is a normal 
 crossings divisor. The corresponding morphism of analytic spaces is denoted 
 $f^{an}: E^{an}\to M^{an}$. We have 
 \[
  H_{DR} ^1 (E^{an}/M^{an}) =  R^1 f^{an}_{\ast}\CC \otimes _{\CC}\mc{O}_{M^{an}}, \quad
  R^1 f^{an}_{\ast} \CC = \mr{Ker}(\nabla ^{an}).
  \]
 When clear in context, we omit the superscript $an$ for a morphism of analytic spaces.
 
 We can describe this Gauss-Manin connection explicitly as follows. 
 Let $U$ be an analytic set homeomorphic with a unit disk in the complex $u$-plane. 
 We consider a family $f: E \to U$ of elliptic curves in Weierstrass form with holomorphic
 functions $g_2(u)$, $g_3(u)$ with  $\Delta (u) = g_2(u)^3 -27g_3(u)^2 \neq 0$ at all points $u \in U$.
 This family can be regarded as the base-change the universal
 $E^{an}\to M^{an}$ by a morphism $U \to M^{an}$. 
 At any given point $u_0 \in U$, $H^1 (E_{u_0}, \CC) =( R^1 f_{\ast}\CC)_{u_0}$, and the 
 elements can be represented by differentials of the first and second kind modulo exact 
 differentials. Thus we can represent the generators of this two-dimensional vector space
 by differential forms $\omega = dx/y$, $\eta = xdx /y$. The Gauss-Manin connection
 gives a lifting of the derivation $\partial/\partial u$ to an endomorphisms of the 
 sheaf  $H_{DR} ^1 (E/U)$. Concretely we extend the action of differentiation by $u$
 to a derivation $D_u$
 of the ring $\mc{O} (U) [x, y]/\langle y^2 -4 x^3 + g_2(u) x + g_3(u)\rangle $ by setting $D_u(x) = 0$.  
 In this way we get a differential equation
 \[
 \frac{d}{du}
 \begin{bmatrix}
 h_1\\ h_2
 \end{bmatrix}=\frac{1}{24\Delta}
 \begin{bmatrix}
 -2\Delta '& 18 \delta\\
 -3g_2 \delta & 2\Delta '
 \end{bmatrix}  \begin{bmatrix}
 h_1\\ h_2
 \end{bmatrix},
\quad \Delta ' = \frac{d \Delta }{du}, \quad
\delta = 3 g_3 \frac{d g_2}{du} - 2 g_2 \frac{d g_3}{du} .
 \]
 Since $E/U$ is topologically trivial, we can choose a 1-cycle 
 $0 \neq \gamma \in H_1 (E_{u_0}, \ZZ)$ which gives a section $\gamma (u)$
 of the sheaf $\underline{H_1 (E, \ZZ)} \sim \underline{\ZZ} ^2 $ on $U$. 
 Then the periods $h_1(u) = \int _{\gamma (u)}\omega$,  $h_2(u) = \int _{\gamma (u)}\eta$ give 
 a basis of local holomorphic solutions to the above differential equation.

 {\bf Example.}  Elliptic curves with $j$-invariant = $j$. Recall that if $j\neq 0, 1$, the most 
 general solution to the equation $g_2 ^3 /\Delta = j$ in any field $k$ of characteristic 
 $\neq 2, 3$ is of the form $g_2 = t \xi ^2$, $g_3 = t \xi ^3$, $t =27 j (j-1)^{-1}$, 
 $\xi \in k^{\times}$. (see \cite[Lemma 1]{yI0}). We can apply this to $k = \QQ(j)$, with 
 $g_2 = g_2 = t$ we get a family of elliptic curves over $\PP^1 _j - \{0, 1, \infty \}$ with 
 $j$-invariant = $j$. We can call this a universal elliptic curve, although strictly speaking 
 it does not represent the obvious functor (for this we need the modular stack or orbifold
 quotient $\slz \backslash \backslash\mfr{H}$). Nonetheless we will consider this 
 family $E \to \PP^1 _j - \{0, 1, \infty \} $ as a scheme over $ A = \ZZ [2^{-1}, 3^{-1}][j, (j(j-1))^{-1}]$. 
 That is 
 \[
 E = \mr{Proj}\left ( A[x, y, z]/\langle y^2 z- 4x^3 +t x z^2 + t z^3\rangle\right ) , \quad 
 t = 27 j (j-1)^{-1}.
 \]

See \cite{St2}.
Consider the hypergeometric differential equation
\[
\frac{d^2\omega}{dx^2} + \frac{1}{x}\frac{d\omega}{dx}
+  \frac{(31/144)x - (1/36)}{x^2(x-1)^2}\omega=0
\]
 In terms of Riemann's P-function, this is 
 
 \[
P
\begin{bmatrix}
0 & \infty &1 &\\
-1/6 & 0 & 1/4 & x\\
1/6& 0&3/4 &
\end{bmatrix}= 
x^{-1/6}(x-1)^{1/4}
P
\begin{bmatrix}
0 & \infty &1 &\\
0& 1/12 &0 & x\\
1/3 &1/12&1/2  &
\end{bmatrix}.
\]
As was known classically, this is the differential equation for the periods in this  family of 
elliptic curves. Moreover, the monodromy matrices around the  singular points 
$x = \infty, 0, 1$ are respectively
\[
\begin{pmatrix} 1 & 1\\ 0 &1 \end{pmatrix}, \quad\begin{pmatrix} 1 & 1\\ -1&0 \end{pmatrix},\quad 
\begin{pmatrix} 0 & -1\\ 1&0 \end{pmatrix}.
\]
These generate $\slz$. In fact, if $y_1(x), y_2(x)$ are two linearly independent (multivalued)
holomorphic solutions to the differential equation, the ratio $y_1(x)/y_2(x)$ is the Schwarzian 
function for this situation, i.e., it is essentially the inverse of the $j$-function, 
$j : \mathfrak{H} \to \CC$.
A classical reference: \cite{Fricke}.
See Stiller's papers for a modern exposition.

{\bf Example.} Modular families. For a congruence subgroup $\Gamma \subset \slz$ we have  a modular 
curve $X_{\Gamma}$ whose analytic space is $\Gamma \backslash \mathfrak{H}$. When 
$\Gamma$ has no nontrivial elements of finite order, there is a universal elliptic 
curve $E_{\Gamma}\to X_{\Gamma}$. If $\Gamma _1 \subset \Gamma _2$ are two such 
subgroups, there is a morphism $u:X_{\Gamma _1}  \to X_{\Gamma _2}$, and a map
\[
\varphi: u^* E_{\Gamma _2} \to E_{\Gamma _1} 
\]
of elliptic curves over $X_{\Gamma _1} $. This is an isogeny, and therefore it induces an isomorphism
of the $\mc{D}$-modules and the $\ell$-adic representations (since we ignore torsion). Thus we 
obtain transformations of the corresponding motivic sheaves. If $\Gamma$ has elliptic points, 
the situation is more complicated. We do not have universal families. For instance, consider
$\Gamma (2) \subset \Gamma (1) = \slz$. The covering of modular curves
is $X(2) = \PP^1 _{\lambda} \to X(1) = \PP^1 _{j} $ given by 
\[
j = \frac{27 \lambda ^2 (\lambda-1)^2}{4(\lambda ^2 - \lambda +1)^3}. 
\]
Pulling back the elliptic curve with $j$-invariant $j$ by this map does not give the universal Legendre curve
$y^2 = x (1-x)(1-\lambda x)$. The corresponding transformation 
of hypergeometric equations has a Kummer twist:

\[
\pFq{2} {1} {\frac{1}{12}, \frac{5}{12}}{,1}  { \frac{27 \lambda ^2 (\lambda-1)^2}{4(\lambda ^2 - \lambda +1)^3} } =
(1-\lambda + \lambda ^2)^{1/4}
\pFq{2} {1}{ \frac{1}{2}, \frac{1}{2}}{1} {\lambda}.
\]

{\bf Example.} The AGM transform.
See \cite{Cox}. Gauss discovered the 
following transformation of elliptic integrals during his investigations of the arithmetic geometric mean (AGM).
Let 
\[
F(k) := \int_{0}^{1}\frac{dx}{\sqrt{(1-x^2)(1-k^2x^2)}}= 
\frac{\pi}{2} \pFq {2}{1} {\frac {1}{2},  \frac {1}{2}} { ,1}{k^2}
\]
then
\[
F\left (\frac{2\sqrt{k}}{1+k}\right) = (1+k)F(k).
\]

Let $J_m $ be the family of  curves
$y^2 = (1-x^2)(1-m^2x^2)$, and define
\[
m = \frac{2\sqrt{k}}{1+k},\quad x = \frac{(1+k)z}{1+kz^2}, \quad
y = \frac{1-kz^2}{(1+kz^2)^2}w := C w, 
\]
then the above equation becomes 
\[
C^2 (w^2 = (1-z^2)(1-k^2z^2)).
\]
Moreover, 
\[
\frac{dx}{y} = (1+k) \frac{dz}{w}.
\]
This can be understood as follows (see \cite{DeRa}). Let
\[
M_4 = \mr{Spec}\ZZ[i, 1/2, \sigma, (\sigma(\sigma ^4 -1))^{-1}].
\]
This is the moduli scheme for $\Gamma (4)\subset \slz$. 

The universal 
elliptic curve for this is 
\[
E_{\sigma}: y^2 =  x(x-1)(x-\lambda), \quad \lambda = (\sigma + \sigma ^{-1})^2/4.
\]
This curve is isomorphic with the Jacobi quartic
\[
C_{\sigma}: y^2 = (1-\sigma ^2 x^2)(1-x^2 /\sigma^2).
\]
via the change of variables
\[
X = \frac{ \sigma ^2+1}{2 \sigma ^2}\cdot\frac{x-\sigma}{x-1/\sigma}, \ \ 
Y = \frac{\sigma ^4-1}{4 \sigma ^3}\cdot\frac{y}{(x-1/\sigma)^2}
\]
(see \cite{Shioda1}).

A rescaling $x \mapsto \sigma x$ gives the equivalent 
curve $y^2 = (1-\sigma ^4 x^2)(1-x^2 )$, which shows that 
we can view this curve as the pull-back 
via the projection 
\[
M _4 \to M_{2, 4} := \mr{Spec}\ZZ[ 1/2, k, (k(k ^2 -1))^{-1}], 
\quad \sigma \mapsto \sigma ^2 = k
\]
of the quartic  $J_{k}: w^2 = (1-z^2)(1-k^2z^2)$ on  $M_{2, 4}$.
$M_{2, 4}$ is the moduli scheme for the group
\[
\Gamma _{2, 4}= \left \{
\begin{pmatrix}
a & b\\c&d
\end{pmatrix}\in \slz \mid a \equiv d \equiv 1 \text {\ mod } 4, 
c \equiv 0 \text{\ mod }4, b \equiv 0 \text{\ mod }2.
\right \}
\]
Clearly $\Gamma (4 ) \subset \Gamma _{2, 4}\subset \Gamma  (2)_0$.
Here $\Gamma (2)_0$ is the subgroup of $\Gamma (2)$ defined
by $a \equiv d \equiv 1 \text{\ mod\ }4$. Then
$\Gamma (2)_0 \cong \Gamma (2) /\pm 1$. The quotient 
$\Gamma (2)_0\backslash\mathfrak{H}$ is the $\lambda $-line. 
We have $\Gamma (2)_0 /\Gamma (4) \cong \ZZ /2 \times \ZZ /2$, 
and  $\Gamma _{2, 4} /\Gamma (4) \cong \ZZ /2$. The map
$\Gamma (4)\backslash\mathfrak{H}\to \Gamma (2)_0\backslash\mathfrak{H}$
is defined by   
\[
\lambda = (\sigma  + \sigma ^{-1})^2/4.
\]
Then $k(\tau)$ is a Hauptmodul for $\Gamma _{2, 4}$. The 
transformation for the AGM is defined by the map $k(\tau) \mapsto k(2 \tau)$. More 
precisely, it is the correspondance defined by the algebraic 
relation relating $k(\tau), k(2 \tau)$, viz., 
\[
k(\tau) = \frac{2\sqrt{k(2 \tau)}}{1+k(2 \tau)}.
\]

Define two function $p(\sigma) = \sigma ^2$, 
$q(\sigma) = 2\sigma /(1+\sigma ^2)$ both 
mapping $M_4 \to M_{2, 4}$. Then 
\[
 X = \frac{(1+\sigma ^2)x}{1+\sigma ^2 x^2}, \quad
  Y = \frac{(1-\sigma ^2x^2)y}{(1+\sigma ^2 x^2)^2}
\]
defines an isogeny $p^* J_k \cong E_{\sigma}\to q^* J_k$. 

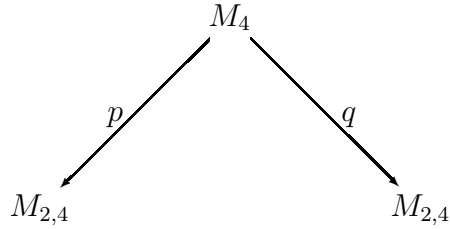
\[
\begin{diagram}
&& M_4   & & \\
& p\ \ \ \  \ldTo\ & & \rdTo\ \ \ \ \ \  q\\
M_{2, 4} & & & & M_{2, 4} 
\end{diagram}
\]

{\bf Example.} The Borwein's cubic transform.

 \begin{equation}\label{2F1-cubic}
\pFq{2}{1}{\frac{1}{3},  \frac{2}{3}}{,1}{1-x^3} 
=
\frac{3}{1+2x} \
\pFq{2}{1}{\frac{1}{3},  \frac{2}{3}}{,1}{\left(\frac{1-x}{1+2x}\right)^3},
\end{equation}
proved by Borwein and Borwein \cite{borweins-1}, \cite{borweins-2} as a cubic analogue of Gauss' quadratic AGM.  
Just as Gauss's formula relates
$\tau$ with $2 \tau$, Borweins' formula relates 
$\tau$ with $3 \tau$.




A finite-field analog of this was proved in \cite{FLRST}:
\begin{theorem}\label{finite2F1-cubic}
For $p\equiv 1 \pmod{3}$ prime, and let $\omega$ be a primitive cube 
root of unity and let $\eta_3$ be a primitive cubic character in $\widehat{\mathbb{F}_p^\times}$.  If $\lambda \in \FF_p$ 
satisfies $1 + 2\lambda \neq 0$,  then
\begin{align*}
\pFFq{2}{1}{\eta_3 & \eta_3^2}{& 1}{1-\lambda^3} 
=
\pFFq{2}{1}{\eta_3 & \eta_3^2}{& 1}{\left(\frac{1- \lambda}{1 + 2\lambda}\right)^3}.
\end{align*}
\end{theorem}
This has the geometric meaning as follows. Let 
$M^0 _3$ be the coarse moduli space for
$\Gamma _0 (3)$. This group is an arithmetic triangle group with two cusps and one elliptic 
point of order 3. The elliptic point means that
there is no universal elliptic curve for 
$\Gamma _0 (3)$, but the family of curves
\[
C: y^2 + x y + (1/27)t y = x^3
\]
is a curve with a rational 3 -torsion point, 
namely $(0, 0)$, which has the correct $j$-invariant. Here $t$ is a Hauptmodul for 
$\Gamma _0 (3)$. Let $M_3$ be the moduli space for 
$\Gamma _3$, the principal congruence subgroup of level 3. The projection $M_3 \to M^0_{3}$
is $z \mapsto z^3$ in suitable coordinates.
Let 
\[
p(z) = 1-z^3, \quad q(z) = \left (\frac{1-z}{1+2z}\right )^3
\]
Then, there is an isogeny $q^* C \equiv p^* C$. In fact, dividing the left-hand side by the subgroup
generated by $(0, 0)$ gives the right-hand side. 
The map
\[
z(\tau) \mapsto \frac{1-z}{1+2z}(3\tau)
\]
is induced by the Atkin-Lehner involution $W_3$.
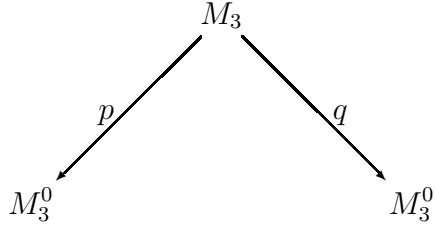
\[
\begin{diagram}
&& M_3   & & \\
& p\ \ \ \  \ldTo\ & & \rdTo\ \ \ \ \ \  q\\
M^0_{3} & & & & M^0_{3} 
\end{diagram}
\]

This transformation can be deduced from another
transformation of a 2-variable hypergeometric, as will be explained in the next section.

\section{The Picard family of curves}
\label{S:picard}
We finish with one example of a two-variable Appell-Lauricella equation. The treatment here is 
only a sketch, with full details to appear elsewhere. Unexplained notation 
is taken from the quoted papers.

The family of quartic curves
\[
y^3 = x(1-x)(1-\lambda x)(1-\mu x)
\]
depending on parameters $\lambda, \mu$ with $\lambda \neq 0, 1$, $\mu \neq 0, 1$, 
$\lambda \neq \mu$ is a family of genus 3 curves whose Jacobian varieties 
have endomorphism rings containing $\ZZ[\omega]$, where
$\omega = \exp(2 \pi i/3)$. These were first studied 
by Picard, see \cite{picard}, \cite{picard2},  \cite{Holz1}, \cite{Holz2}, \cite{Shiga}. The space

\[
\CC ^2 -  \{\lambda = 0, \lambda = 1, \mu = 0, \mu = 1, \lambda = \mu \} 
= \Gamma \backslash \mathbb{B}_2
\]
is the set of $\CC$-points of the Shimura variety of PEL type representing principally polarized abelian 3-folds with an 
embedding of $\ZZ[\omega]$ into the endomorphism algebra. Here $ \mathbb{B}_2\subset \CC^2$
is the open unit ball, and $\Gamma \subset \mathrm{SU}(2, 1;\ZZ[\omega])$ is 
the congruence subgroup of the points of unitary group with signature $(2, 1)$ with coordinates in the 
Eisenstein integers which satisfy 
\[
\gamma \equiv 1 \mr{\ mod\ } (1 - \omega).
\]
The Jacobians of the Picard family form the universal family of abelian varieties over this space. 

The periods of integrals in this family satisfy an Appell-Lauricella differential equation
(\cite{appell}, \cite{appell2}, \cite{appell-kampe}). Remarkably, 
this was shown by Picard, who also computed the monodromy, in effect
discovering the Picard-Lefschetz formula. He observed that the 
monodromy preserved a Hermitian form of signature $(2, 1)$. 

 In \cite{koike-shiga1}, Koike and Shiga studied Appell's $F_1$-hypergeometric function in two variables to establish a new three-term arithmetic geometric mean result (AGM), related to Picard modular forms.  

Let $x,y \in \mathbb{C}$, and let $\omega$ be a primitive cubic root of unity.  Then
\begin{theorem}\label{cubic} 
\begin{multline}\label{F1-cubic}
F_1\left[ \frac{1}{3};\, \frac{1}{3},\,  \frac{1}{3};\, 1 \, \Big|\, 1-x^3,\,1-y^3\right]  \\
=
\frac{3}{1+x+y} \
F_1\left[ \frac{1}{3};\, \frac{1}{3},\, \frac{1}{3};\,1 \, \Big|\,
	\left(\frac{1+\omega x + \omega^2 y}{1 + x + y}\right)^3,\,
	\left(\frac{1+\omega^2 x + \omega y}{1 + x + y}\right)^3\, \right].
\end{multline}
\end{theorem}
The following is a finite-field analog of the formula 
of Koike and Shiga, proved in \cite{FST}.
\begin{theorem}\label{finiteF1-cubic} 
Let $p\equiv 1 \pmod{3}$ be prime, let $\omega$ be a primitive cubic root of unity, and let $\eta_3$ be a primitive cubic character in $\widehat{\mathbb{F}_p^\times}$.  If $\lambda,\mu \in \FF_p$ satisfy $1 + \lambda + \mu \neq 0$,  then
\begin{align*}
\FAFn{2}{\eta_3}{\eta_3 &  \eta_3}{1}{1-\lambda^3, 1-\mu^3} 
=
\FAFn{2}{\eta_3}{\eta_3 &  \eta_3}{1}
	{\left(\frac{1+\omega \lambda + \omega^2 \mu}{1 + \lambda + \mu}\right)^3, \left(\frac{1+\omega^2 \lambda + \omega \mu}{1 + \lambda + \mu}\right)^3\, }. 
\end{align*}
\end{theorem}
The key observation in \cite{FST} is that a truncation of the Appell-Lauricella series gives 
the Hasse invariant of the Picard curve. The equality of the finite field analog reduces to 
comparing the number of rational points of the 
two transforms of the Picard curves. 

The main result of \cite{koike-shiga1}
can be formulated as follows: There is a 
diagram
\[
\begin{diagram}
&& M_{\theta}   & & \\
& p\ \ \ \  \ldTo\ & & \rdTo\ \ \ \ \ \  q\\
M_{\sqrt{-3}} & & & & M_{\sqrt{-3}} 
\end{diagram}
\] 
where $M_{\sqrt{-3}}$ is a compactification of the  PEL Shimura variety 
attached to the group
\[
\Gamma (\sqrt{-3}) =\{ 
\gamma \in 
\mr{SU}(2, 1; \Z[\omega]) \mid
\gamma \equiv\  1 \ \mr{  mod\ }
(1 - \omega)
\}.
\]
 An open subset $M ^{\circ}_{\sqrt{-3}}$
 of the $\CC$-points of this is the quotient
 $\Gamma (\sqrt{-3})\backslash \mathbb{B}_2$. In fact 
 $M _{\sqrt{-3}}  = \mathbb {P}^2$
with coordinates $\xi _0,\xi _1, \xi_2 $. 
We have $\xi _{\mu} = \theta _{\mu}(u, v)^3$
for certain explicit theta functions 
depending on $(u, v)\in \mathbb{B}^2$.
The rational map 
$\mathbb{P}^2 \to \mathbb{P}^2$ given by 
$(\lambda _0,\lambda _1 , \lambda_2) \to 
(\lambda _0 ^3= \xi _0,\lambda _1 ^3 = \xi _1 ,\lambda _2 ^3 = \xi _2 )$ corresponds 
to a congruence subgroup
\[
\Gamma (\theta) \subset \Gamma (\sqrt{-3})
\]
of index 9. That is, there is a compactification of  
$\Gamma (\theta )\backslash \mathbb{B}_2$
which is $\mathbb{P}^2$ with coordinates $(\lambda _0, \lambda _1, \lambda _2)$. 
This is denoted $ M_{\theta} $ in the diagram above.

The Jacobians of the Picard curve
\[
C(\xi): y^3 = 
x (x-\xi_0)(x-\xi _1)(x-\xi _2)
\]
form the universal abelian variety 
$\mathcal{A} (\xi)$ over an open set 
$M_{\sqrt{-3}} ^{\circ}\subset M_{\sqrt{-3}}$. 
We can write this family as 
$\mathcal{A} (u, v)$ to emphasize its dependence on $(u, v) \in \mathbb{B}_2$.

One of the main results of \cite{koike-shiga1} is 

\begin{theorem}
There is an 
isogeny of degree 27, 
$a: q^* \mathcal{A} \to p^* \mathcal{A}$ covering  the map 
$(u, v)\mapsto (\sqrt{-3}u, 3 v)$ of $ \mathbb{B}_2$.

In affine coordinates $x = \xi _1 /\xi_0, y = \xi_2/\xi_0$ and
$w = \lambda _1/\lambda _0, z = \lambda _2/\lambda _0$ the maps are given by 
\begin{align*}
(x, y) = p(w, z) &= (1-w^3, 1-z^3) \\ 
(x, y) = q(w, z) &= \left (
\left (\frac{1 + \omega w + \omega ^2 z}
{1+ w + z}\right )^3, 
\left (\frac{1 + \omega^2 w + \omega z}
{1+ w + z}\right ) ^3
\right )
\end{align*}
\end{theorem}
This is their isogeny formula, which is deduced
from transformation properties of theta functions. Since the Appell-Lauricella 
differential equation is the DE for the periods of the Picard family, this isogeny
formula essentially proves \ref{cubic}. The entire picture above should be valid
as schemes over $\ZZ[\omega, 1/3]$ and for that reason, taking 
$\ell$-adic coefficients, one obtains theorem \ref{finiteF1-cubic}. For this one needs the theory of compactified 
PEL Shimura varieties over integer rings. 
Also note 
that this gives a transformation formula 
for the corresponding $p$-adic Appell-Lauricella, via the machinery of 
crystalline cohomology.

\section{Appendix: Local systems }
\label{S:app1}

\subsection{ $\CC$-local systems and differential equations}
\label{SS:app1}

If $X$ is a nonsingular algebraic variety over $\CC$ we let $X^{\mr{an}} = 
X(\CC)$ be the set of complex points with the classical topology. For simplicity, assume
that $ X^{\mr{an}}$ is connected.

Recall the following dictionary:
The  following categories are equivalent:

\begin{itemize}
\item[D1.] Local systems of finite-dimensional $\CC$-vector spaces {\sf V }on 
 $ X^{\mr{an}}$.
 \item[D2.] Representations $\rho : \pi _1 (X^{\mr{an}},  x) \to \mr{GL}_n(V)$
on finite-dimensional $\CC$-vector spaces $V$.
\item[D3.] Holomorphic integrable connections 
 \[
\nabla : \mc{V}^{an} \to \Omega^ 1 _{X^{\mr{an}}/\CC} \otimes _{\mc{O}_X^{\mr{an}}} \mc{V}^{an}
\] 
where $\mc{V}\an$ is  a locally free $\mc{O}_X^{\mr{an}}$ sheaf of finite rank.  
 \item[D4.] Integrable algebraic connections 
 \[
\nabla : \mc{V}\to \Omega^ 1 _{X/\CC} \otimes _{\mc{O}_{X}} \mc{V}
\] 
where $\mc{V}$ is  a locally free $\mc{O}_X$ sheaf of finite rank, and which have regular 
singular points ``at infinity''.   
 \end{itemize}
Some comments:
\begin{itemize}
\item[1.] The {\it morphisms} in each of these categories are the obvious ones. 
\item[2.] The integrability condition is that the composed map (curvature) 
\[
 \mc{V}\to \Omega^ 1 _{X/\CC} \otimes _{\mc{O}_{X}} \mc{V}\to
  \Omega^ 1 _{X/\CC} \otimes _{\mc{O}_{X}}   \Omega^ 1 _{X/\CC} \otimes _{\mc{O}_{X}} \mc{V} 
  \to    \Omega^ 2 _{X/\CC}      \otimes _{\mc{O}_{X}}      \mc{V}  \]
  is 0. We then get both algebraic and holomorphic deRham complexes
  \[
 \Omega^ {*} _{X/\C} \otimes _{\mc{O}_{X}} \mc{V}. \]
 In the analytic case, this deRham complex is a resolution of the sheaf ${\sf V}$,
 by the holomorphic Poincar\'e lemma. 
 \item[3.] We call these connections differential equations. D3 is essentially due to 
 Frobenius. D4 is called the Riemann-Hilbert correspondence.
 
 \item[4.] Regular singular points means this: Let $X \subset \bar{X}$
 be a compactification such that $\bar{X} - X = D$ is a divisor with normal crossings
 (exists by Hironaka's theorem).  Then there is a locally free sheaf $\bar{\mc{V}}$ on 
 $\bar{X}$ extending $\mc{V}$ and a connection
 \[
 \bar{\nabla} :  \bar{\mc{V}}\to \Omega^ 1 _{X/\CC}(\log D) \otimes _{\mc{O}_{X}} \bar{\mc{V}} \]
 extending $\nabla$. When $\dim X = 1$ this is equivalent to Fuchs' growth conditions
 at singular points of the differential equations.
 
 \item[5.] In the language of $\mc{D}_X$-modules, 
 connection with regular singular points = regular holonomic 
 $\mc{D}_X$-module which is coherent (hence locally free) as
 an $\mc{O}_X$-module. 
\end{itemize}

The functors go like this:

$1\Rightarrow 2$: ${\sf V} \mapsto {\sf V}_x$ which is a  $\pi _1 (X^{\mr{an}},  x)$-module. 

$1\Rightarrow 3$: ${\sf V}\mapsto \mc{V} = {\sf V}\otimes _{\underline{\CC}}\mc{O}_{X\an}$
with connection $\nabla = 1 \otimes d$.

$4\Rightarrow 3$: $ \mc{V} \mapsto \mc{V} \otimes _{\mc{O}_X}\mc{O}_{X\an} = \mc{V}\an$, 
with the obvious connection.
 A proof of the regularity theorem can be found in \cite{DelDE}.

\subsection{ $\ell$-adic local systems}
\label{SS:app2}

A reference: \cite{DelWeil2}. Fix  a prime number $\ell$.  In this section: scheme = 
a separated noetherian scheme on which $\ell$ is invertible. 
 We are interested in constructible $\bar{\QQ}_{\ell}$-sheaves on $X$, in particular, 
 those that are lisse. In this section: the \'etale topology is understood.
 
 An $\ell$-adic representation of a profinite group $\pi$ on a  $\bar{\QQ}_{\ell}$-vector 
 space $V$  is a homomorphism $$\sigma : \pi \to \mr{GL}(V)$$ such that 
 there is a finite subextension $E/\QQ _{\ell}$ and an $E$-structure $V_E$ on $V$
 such that $\sigma $ factorizes in a continuous homomorphism $\pi \to \mr{GL}(V_E)$.

  Recall that a geometric point $\bar{x}$ of a scheme $X$ is a morphism of the 
  spectrum of an algebraically closed field denoted $k(\bar{x})$. It is localized 
  in $x \in X$ if its image is $x$. 
  
  If $X$ is connected and pointed by a geometric point $\bar{x}$, the functor
  \[
  \mc{F} \mapsto \text{the \ } \pi_1 (X, \bar{x}) -\text{\ module\ } \mc{F}_{\bar{x}}
  \]
  is an equivalence of categories between the categories of
  \begin{itemize}
 \item[1.] lisse  $\bar{\QQ}_{\ell}$-sheaves on $X$;
 \item[2.] $\ell$-adic representations of $\pi _1(X, \bar{x})$.
   \end{itemize}
Here   $\pi _1(X, \bar{x})$ is Grothendieck's fundamental group. Especially if $X = \mr{Spec}(k)$ is a field, the category of lisse  $\bar{\QQ}_{\ell}$-sheaves on $X$ is 
 equivalent to the category of $\ell$-adic representations of $\mr{Gal}(\bar{k}/k)$.

\section{Appendix: Motivic Sheaves }
\label{S:app3} 

For an introduction to the theory of motives, see Andr\'e's book \cite{andre}.
Fix a field $k$. The idea is to construct a rigid $\otimes$-category of mixed 
motives $\mbf{MM} (k)_F$ with coefficients in $F$,  a field of characteristic 0, together 
with realization functors into various cohomology theories (Betti, deRham, 
$\ell$-adic \'etale, $p$-adic crystalline).  The pure objects should constitute 
a semisimple subcategory $\mbf{M}(k)_F$ which is Grothendieck's  category of motives
for numerical equivalence. For a smooth projective variety $X$, an 
idempotent $e \in \mr{Corr}(X)_F$, the correspondence ring of $X$, and an integer $i$, we have an element  
$eh(X)(i)\in M(k)_F$.  We think of $h(X)$ as representing the cohomology of $X$, and the integer $i$ represents a Tate twist. The idempotent 
could be for instance a K\"unneth projector onto a factor $h^j(X)$. In any realization, this becomes
(e.g., for \'etale cohomology)  $eH_{et }(X\otimes _k \bar{k}, \overline{\QQ}_{\ell})(i)$,

What is usually constructed is a triangulated  $\otimes$-category $\mbf{DM}(k)$. The realization functors then map 
to various derived categories (e.g., of $ \overline{\QQ}_{\ell}- \mr{Gal}(\bar{k}/k)$-vector spaces). One hopes 
for a $t$-structure on $\mbf{DM}(k)$ whose heart is $\mbf{MM} (k)_F$, and such that 
$\mbf{DM}(k) = D^b (\mbf{MM} (k)_F)$. The situation today is that there are various constructions of 
$\mbf{DM}(k)$ (Voevodsky, Hanamura, Levine, Nori) but the existence of a $t$-structure is a conjecture.

One extends this construction to motivic sheaves. Given a scheme $S$, there is a  $\otimes $-triangulated category 
$\mbf{DA}(S)$ of motivic sheaves on $S$. For instance, given a morphism $f : X \to S$, one wishes 
for objects $f_{\ast } \QQ_X(0)\in \mbf{DA}(S)$, where $\QQ_X(0) \in \mbf{DA}(X)$ is the unit element. 
Then the cohomological realizations (e.g., for $\ell$-adic cohomology) should be 
$\mbf{R}f_{\ast } \bar{\QQ}_{X, \ell} \in D^b (S, \bar{\QQ}_{\ell})$. When 
$S =\mr{Spec}(k)$, we have $\mbf{DA}(S) =\mbf{DM}(k) $.

There are at least two formalisms of motivic sheaves available, one
closely related to Voevodsky's triangluated category of mixed motives, one related to
Nori's motives, see Arapura's  paper \cite{ara}. See also the papers of Huber, \cite{Huber95}. \cite{Huber95ii}. 
We will follow the exposition in Ayoub's papers, to which we refer the 
reader for further information. This whole theory is built upon Voevodsky's theory of triangulated 
categories of mixed motives, see \cite{morel-voevodsky}.

\subsection{{}\! \! }
\label{SS:app3i}
Let $X$ be a noetherian scheme. 
There is a tensor triangulated category
$\mathbf{DA}(X)$ whose objects will be called \emph{relative motives} over
the scheme $X$. We briefly recall its construction. Let 
$\mr{Sm}/X$ be the category of smooth $X$-schemes of finite type, endowed with 
the \'etale topology. We let $\mbf{Shv}(  \mr{Sm}/X  )$ be the category of 
\'etale sheaves of $\QQ$-vector spaces on  $\mr{Sm}/X $. Given a smooth 
$X$-scheme $Y \to X$, we let $\QQ _{et} (Y) = \QQ _{et} (Y\to X)$ be the \'etale sheaf associated to the 
presheaf defined by 
\[
\QQ (Y) (-) := \QQ (\Hom _{  \mr{Sm}/X }(  -, Y    )   ).
\]
The category $\mathbf{DA}(X)$ is defined in two steps:

\begin{itemize}
\item[1.] The category of effective motives $\mathbf{DA }_{\mr{eff}}(X)$ is defined 
as the Verdier quotient of the derived category $\mbf{D} (\mbf{Shv}(  \mr{Sm}/X  ))$
by the smallest triangulated category closed under infinite sums and containing all
complexes $[\QQ _{et} (\mathbb{A}^1_Y)\to \QQ _{et} (Y)]$. The object $\QQ _{et} (Y) $, 
viewed as an element of $\mathbf{DA }_{\mr{eff}}(X)$, is denoted $\rm{M} _{\mr{eff}}(Y)$.
It is called the effective homological motive associated to $Y\to X$. We denote 
$\rm{M} _{\mr{eff}}(\rm{id}_X: X \to X)$ by  $\un_X$. It is the unit for tensor product
in $\mathbf{DA }_{\mr{eff}}(X)$.

\item[2.] $\mathbf{DA}(X)$ is obtained from $\mathbf{DA }_{\mr{eff}}(X)$ by formally 
inverting the operation $T_X \otimes -$, where $T_X$ is the Tate object. This is defined as
\[
T_X = \mr{ker} \left ( [\QQ _{et} (\mathbb{A}^1_X - o(X)\to X)\to \
\QQ _{et} (\mr{id}_X : X\to X)
\right ),
\]
where $o(X)$ is the zero section of $\mathbb{A}^1_X$. Note that the Tate motive 
is defined as $\QQ _X (1) := T_X [-1] $.  The tensor product on $\DM(X)$  makes it a
closed monoidal symmetric category
with unit object $\un_X$. 
\end{itemize}
It can be shown
that, for $X=\Spec(k)$ the spectrum of a perfect
field, we have an equivalence of categories $\DM(k)\simeq
\mathbf{DM}(k)$, where $\mathbf{DM}(k)$ is Voevodsky's category of mixed
motives with rational coefficients.

There is a variant of the above where the sheaves of $\QQ$-vector spaces are 
replaced by $\Lambda$-modules, notation: $\mathbf{DA}(X,  \Lambda)$. The important case for us is when 
$\QQ \subset \Lambda$. The unit object is also denoted $\Lambda _X (0)$, and the Tate objects
$\Lambda _X (n)$.

In \cite{ayoub-these-I, ayoub-these-II}, it is shown that one has the full machinery
of Grothendieck's six operations on the triangulated
categories $\mathbf{DA}(X)$.  Tensor product and Hom, $\otimes _X$ and $\underline{\Hom}_X$;
\[
\text {\ inverse \ and \ direct \ image: \ }
 f^*, f_{*}, \quad \text{for \ } f: X\to Y
\text {\ a \ morphism \ of \ noetherian \ schemes;}
\]
and 
\[
\text { compact \ supports: \ }
 f_{!}, f^{!}, \quad f: X\to Y
\text {\ a \ quasi-projective\ \ morphism \ of \ noetherian \ schemes,}
\]
as well as nearby and vanishing cycle functors.
Moreover, it was shown in \cite{realiz-oper, realiz-etale} that there are realization functors, compatible with the above functors. 

{\bf Betti.} Let $k$ be a field and $X$ a scheme of finite type over $k$. Let  $\sigma : k \hookrightarrow \CC$ be an embedding. Then there is symmetric monoidal unitary functor
\[
{\sf Betti}_{X, \sigma}  : \mathbf{DA}(X) \to \mathbf{D}(X^{\mr{an}})
\]
which commutes in the obvious sense with the above functors when restricted to 
{\it compact objects}, e.g., if $f: Y \to X$ is a morphism 
of finite type of quasi-projective $k$-schemes of finite type, then there are natural isomorphisms
\[
(f^{\mr{an}})^*\circ {\sf Betti}_{X, \sigma} \cong  {\sf Betti}_{Y, \sigma} \circ f^*.
\]
The triangulated subcategory of compact objects $\mathbf{DA}_{\mr{cp}}(X)$ is generated by 
the quasi-projective $Y \to X$. On that subcategory, we have an isomorphism
\[
{\sf Betti}_{X, \sigma} (\underline{{\sf Hom}}(A, B)) \cong 
 \underline{{\sf Hom}}({\sf Betti}_{X, \sigma}(A), {\sf Betti}_{X, \sigma}(B)).
\]

{\bf Hodge-deRham.}  For the precise statements, see \cite{Iv3}.

{\bf \'Etale.} 
See \cite{realiz-etale}. Let $E/\QQ$ be a finite extension field. For each prime number $\ell$ there is a functor
\[
{ \mathfrak R} ^{et} _{S, \ell}:  \mathbf{DA}^{et} _{ct}(S,E) \rightarrow  \mathbf{D}^{et} _{ct}(S,E \otimes \QQ _{\ell}) 
\]
from the category of constructible motives on $S$ with $E$-coefficients, to the derived category of 
constructible $E \otimes \QQ _{\ell}$-adic sheaves on $S$. The validity of this theorem depends on certain 
broad technical hypotheses on $S$, which are valid for all the schemes appearing in this paper.
This functor is compatible with the 6 operations above, as well as nearby and vanishing cycle sheaves. 
See also \cite{Iv1}, \cite{Iv2}.

{\bf Crystalline.} 
This is not yet available.

Currently under development, there is also a theory of perverse objects, see \cite{IvMor}.

\vskip .5 cm
In this paper, $\Lambda = \Lambda _N =  K_N = \QQ (\mu _N)$. If $G$ is a finite group acting on a motivic 
sheaf $M$ over any scheme in which $\# G$ is invertible, then for any idempotent $e$ in the group-ring $\Lambda [G]$ there is an image $eM$. If 
$\chi$ the character of a irreducible representation and $e = (1 / \# G)\sum _{g \in G}\chi ^{-1}(g).g$, then 
$eM $ is denoted $M^{\chi}$. 

This paper makes use of the Kummer motives $K(\chi)$ on $\mathbf{G}_m$, 
attached to a character $\chi :G \to \Lambda _N ^{\times}$. We consider the 
\'etale covering  $[N]: \mathbf{G}_m \to \mathbf{G}_m$,  where $\mathbf{G}_m$ is viewed as a scheme over 
$\mr{Spec}(R_N)$, with $R_N = \ZZ[\zeta _N, 1/N]$, i.e., $\mathbf{G}_m = \mr{Spec} R_N [t, t^{-1}].$ $G$ is the Galois group of this covering, which 
may be canonically identified with $\mu _N  = \mu _N(R_N)$ via Kummer theory: if $t$ is the  coordinate
on $ \mathbf{G}_m$, then for any $\sigma \in G$, $\sigma\  t^{1/N} = \zeta _{\sigma} \ t^{1/N}  $  for a 
root of unity $\zeta _{\sigma}$, independent of the choice of $t^{1/N}$.

The Kummer motive is defined by the formula
\[
K(\chi) = \left (  [N]_{\ast } \Lambda _{\mathbf{G}_m} (0) \right ) ^{\chi}
\]
in $ \mathbf{DA}(\mathbf{G}_m, \Lambda_N)$. We have 
\[
[N]_{\ast } \Lambda _{\mathbf{G}_m} (0) = \bigoplus _{\chi : G\to \Lambda_N ^{\times} } K(\chi).
\]
We have the 
realizations:

{\bf Betti.} For each embedding $\varphi : R_N \to \CC$ we get an isomorphism
$\varphi :G =  \mu_N (\Lambda_N) \cong \mu_N (\CC)$.
Then $K(\chi)_{\varphi, B}$ is the $\CC$-local system on the analytic space $\CC^{\times}$ defined by the character
\[
\varphi \circ \chi: \pi _1 (\CC^{\times}, 1) = \ZZ \to \CC^{\times}: k \mapsto \varphi (\chi (\varphi ^{-1} (\exp (2 \pi i k /N)))).
\]

{\bf Hodge-deRham.} $H^1 _{dR}( \mathbf{G}_m/R_N) $ is the $H^1$ of the complex 
$[d: \mc{O}_ { X}\to \Omega _{X /R_N} ^1]$, $X = \mathbf{G}_m $. This is a free $R_N$-module
generated by $\frac{dt}{t}$. Then we have the Gauss-Manin connection 
\[
\nabla: N_{\ast} \mc{O}_ {X} \to  N_{\ast} \mc{O}_ {X} \otimes \Omega _{Y/R_N} ^1, 
\]
where $N : X = \mathbf{G}_m  \to Y =  \mathbf{G}_m  $ is the map $s \mapsto s^N = t$. 
We have 
\[
N_{\ast} \mc{O}_ {X}  =  \bigoplus _{i \in \ZZ/N} \mc{O}_Y s^i =   
\bigoplus _{\chi \in\widehat{ \mu _N}}\left (N_{\ast} \mc{O}_ {X}  \right )^{\chi}.
\]
The identification $i \leftrightarrow\chi $ is given by Kummer theory. For a fixed character $\chi$, there
is a unique $i \in \ZZ/N$ such that $\sigma \ s^i = \chi (\sigma) s^i$, for all $\sigma \in G$.

Now for any $i$ we define a connection on the free rank 1 module  $\mc{O}_Y s^i $ by the formula
\[
\nabla: \mc{O}_Y s^i \to  \mc{O}_Y s^i \otimes \Omega _{Y/R_N} ^1, \nabla (f s^i) =
\left  (t \frac{df}{dt}+ \frac{i}{N}\right )\frac{dt}{t}\otimes s^i
\]
which follows from $d s^i = (i/N) s^i dt/t$.

{\bf \'Etale.} 
We have a canonical epimorphism $\pi _1 \to \mr {Gal}(\QQ (\zeta _N, s)/\QQ(\zeta _N, t))= G$
where  $\pi _1 = \pi _1 ( \mathbf{G}_m, \bar{\eta})$ is Grothendieck's fundamental group, 
$\bar{\eta} = \mr{Spec}(\overline{\QQ (\zeta _N, t)})$. Choose an embedding $\varphi : \mu _N\to \bar{\QQ}_{\ell}^{\times}$
 for a prime number $\ell$ prime to $N$.
Composing the above epimorphism with $\varphi \circ \chi $ we get a character of $\pi _1 $, which
defines the $\ell$-adic local system $K(\chi)_{\varphi, \ell}$.

In our application, we will need the Frobenius traces of the Kummer sheaves. If $t \in \mathbf{G}_m (\mathbf{F}_q)$
is a point $q \equiv 1 $ mod $N$, then
\[
\mr{Tr} (\mr{Frob} _t \mid  (K(\chi)_{\varphi, \ell})_{\bar{t}} ) = \varphi (\chi (t ^{(q-1)/N})).
\]
Note that $t \mapsto t ^{(q-1)/N}$ which sends $\mathbf{F}_q ^{\times} \to \mu _N$ is the character 
giving the canonical action of Frobenius on the Kummer extension: 
\[
\mr{Frob}_q (\sqrt[N]{t}) = t ^{(q-1)/N) }\sqrt[N]{t}.
\]



\bibliographystyle{plain}
\bibliography{DEtransfBib}

\begin{thebibliography}{10}

\bibitem{SGA1}
{\em Rev\^{e}tements \'{e}tales et groupe fondamental ({SGA} 1)}, volume~3 of
  {\em Documents Math\'{e}matiques (Paris) [Mathematical Documents (Paris)]}.
\newblock Soci\'{e}t\'{e} Math\'{e}matique de France, Paris, 2003.
\newblock S\'{e}minaire de g\'{e}om\'{e}trie alg\'{e}brique du Bois Marie
  1960--61. [Algebraic Geometry Seminar of Bois Marie 1960-61], Directed by A.
  Grothendieck, With two papers by M. Raynaud, Updated and annotated reprint of
  the 1971 original [Lecture Notes in Math., 224, Springer, Berlin; MR0354651
  (50 \#7129)].

\bibitem{andre}
Yves Andr\'{e}.
\newblock {\em Une introduction aux motifs (motifs purs, motifs mixtes,
  p\'{e}riodes)}, volume~17 of {\em Panoramas et Synth\`eses [Panoramas and
  Syntheses]}.
\newblock Soci\'{e}t\'{e} Math\'{e}matique de France, Paris, 2004.

\bibitem{appell-kampe}
P.~Appell and J.K. de~F{\'e}riet.
\newblock {\em Fonctions hyperg{\'e}om{\'e}triques et hypersph{\'e}riques:
  Polyn\^{o}mes d'Hermite}.

\bibitem{appell}
Paul Appell.
\newblock Sur les fonctions hyperg\'{e}om\'{e}triques de deux variables.
\newblock {\em Journal de Math\'{e}matiques Pures et Appliq\'{u}ees}, pages
  173--216, 1882.

\bibitem{appell2}
Paul Appell.
\newblock {\em Sur les Fonctions hyp\'{e}rg\'{e}ometriques de plusieurs
  variables les polynomes d'Hermite et autres fonctions sph\'{e}riques dans
  l'hyperspace}.
\newblock Gauthier-Villars, 1925.

\bibitem{ara}
Donu Arapura.
\newblock An abelian category of motivic sheaves.
\newblock {\em Adv. Math.}, 233:135--195, 2013.

\bibitem{ayoub-these-I}
Joseph Ayoub.
\newblock Les six op\'{e}rations de {G}rothendieck et le formalisme des cycles
  \'{e}vanescents dans le monde motivique. {I}.
\newblock {\em Ast\'{e}risque}, (314):x+466 pp. (2008), 2007.

\bibitem{ayoub-these-II}
Joseph Ayoub.
\newblock Les six op\'{e}rations de {G}rothendieck et le formalisme des cycles
  \'{e}vanescents dans le monde motivique. {II}.
\newblock {\em Ast\'{e}risque}, (315):vi+364 pp. (2008), 2007.

\bibitem{realiz-oper}
Joseph Ayoub.
\newblock Note sur les op\'{e}rations de {G}rothendieck et la r\'{e}alisation
  de {B}etti.
\newblock {\em J. Inst. Math. Jussieu}, 9(2):225--263, 2010.

\bibitem{realiz-etale}
Joseph Ayoub.
\newblock La r\'{e}alisation \'{e}tale et les op\'{e}rations de {G}rothendieck.
\newblock {\em Ann. Sci. \'{E}c. Norm. Sup\'{e}r. (4)}, 47(1):1--145, 2014.

\bibitem{BCM}
Frits Beukers, Henri Cohen, and Anton Mellit.
\newblock Finite hypergeometric functions.
\newblock {\em Pure Appl. Math. Q.}, 11(4):559--589, 2015.

\bibitem{borweins-1}
J.~M. Borwein and P.~B. Borwein.
\newblock A remarkable cubic mean iteration.
\newblock In {\em Computational methods and function theory ({V}alpara\'{\i}so,
  1989)}, volume 1435 of {\em Lecture Notes in Math.}, pages 27--31. Springer,
  Berlin, 1990.

\bibitem{borweins-2}
J.~M. Borwein and P.~B. Borwein.
\newblock A cubic counterpart of {J}acobi's identity and the {AGM}.
\newblock {\em Trans. Amer. Math. Soc.}, 323(2):691--701, 1991.

\bibitem{Cox}
David~A. Cox.
\newblock The arithmetic-geometric mean of {G}auss.
\newblock {\em Enseign. Math. (2)}, 30(3-4):275--330, 1984.

\bibitem{Del2}
P.~Deligne.
\newblock Courbes elliptiques: formulaire d'apr\`es {J}. {T}ate.
\newblock In {\em Modular functions of one variable, {IV} ({P}roc. {I}nternat.
  {S}ummer {S}chool, {U}niv. {A}ntwerp, {A}ntwerp, 1972)}, pages 53--73.
  Lecture Notes in Math., Vol. 476, 1975.

\bibitem{SGA4.5}
P.~Deligne.
\newblock {\em Cohomologie \'{e}tale}, volume 569 of {\em Lecture Notes in
  Mathematics}.
\newblock Springer-Verlag, Berlin, 1977.
\newblock S\'{e}minaire de g\'{e}om\'{e}trie alg\'{e}brique du Bois-Marie SGA
  $4\frac{1}{2}$.

\bibitem{DM}
P.~Deligne and G.~D. Mostow.
\newblock Monodromy of hypergeometric functions and nonlattice integral
  monodromy.
\newblock {\em Inst. Hautes \'{E}tudes Sci. Publ. Math.}, (63):5--89, 1986.

\bibitem{DeRa}
P.~Deligne and M.~Rapoport.
\newblock Les sch\'{e}mas de modules de courbes elliptiques.
\newblock In {\em Modular functions of one variable, {II} ({P}roc. {I}nternat.
  {S}ummer {S}chool, {U}niv. {A}ntwerp, {A}ntwerp, 1972)}, pages 143--316.
  Lecture Notes in Math., Vol. 349, 1973.

\bibitem{DelDE}
Pierre Deligne.
\newblock {\em \'{E}quations diff\'{e}rentielles \`a points singuliers
  r\'{e}guliers}.
\newblock Lecture Notes in Mathematics, Vol. 163. Springer-Verlag, Berlin-New
  York, 1970.

\bibitem{DelWeil2}
Pierre Deligne.
\newblock La conjecture de {W}eil. {II}.
\newblock {\em Inst. Hautes \'{E}tudes Sci. Publ. Math.}, (52):137--252, 1980.

\bibitem{DK}
Igor~V. Dolgachev and Shigeyuki Kondo.
\newblock Moduli of {$K3$} surfaces and complex ball quotients.
\newblock In {\em Arithmetic and geometry around hypergeometric functions},
  volume 260 of {\em Progr. Math.}, pages 43--100. Birkh\"{a}user, Basel, 2007.

\bibitem{ESY}
H\'{e}l\`ene Esnault, Claude Sabbah, and Jeng-Daw Yu.
\newblock {$E_1$}-degeneration of the irregular {H}odge filtration.
\newblock {\em J. Reine Angew. Math.}, 729:171--227, 2017.
\newblock With an appendix by Morihiko Saito.

\bibitem{MF}
Matthias Flach.
\newblock Periods and special values of the hypergeometric series.
\newblock {\em Math. Proc. Cambridge Philos. Soc.}, 106(3):389--401, 1989.

\bibitem{FST}
Sharon Frechette, Holly Swisher, and Fang-Ting Tu.
\newblock A cubic transformation formula for {A}ppell-{L}auricella
  hypergeometric functions over finite fields.
\newblock {\em Res. Number Theory}, 4(2):Art. 27, 27, 2018.

\bibitem{FSY}
Javier {Fres{\'a}n}, Claude {Sabbah}, and Jeng-Daw {Yu}.
\newblock {Hodge theory of Kloosterman connections}.
\newblock {\em arXiv e-prints}, page arXiv:1810.06454, October 2018.

\bibitem{Fricke}
Robert Fricke.
\newblock {\em Die elliptischen {F}unktionen und ihre {A}nwendungen. {E}rster
  {T}eil. {D}ie funktionentheoretischen und analytischen {G}rundlagen}.
\newblock Springer, Heidelberg, 2011.
\newblock \copyright 2012, Reprint of the 1916 original, With a foreword by the
  editors of Part III: Clemens Adelmann, J\"{u}rgen Elstrodt and Elena
  Klimenko.

\bibitem{LF}
Lei Fu.
\newblock {$\ell$}-adic {GKZ} hypergeometric sheaves and exponential sums.
\newblock {\em Adv. Math.}, 298:51--88, 2016.

\bibitem{FWZ}
Lei {Fu}, Daqing {Wan}, and Hao {Zhang}.
\newblock {The $p$-adic Gelfand-Kapranov-Zelevinsky hypergeometric complex}.
\newblock {\em To Appear}, page arXiv:1804.05297, April 2018.

\bibitem{FLRST}
Jenny Fuselier, Ling Long, Ravi Ramakrishna, Holly Swisher, and Fang-Ting Tu.
\newblock Hypergeometric functions over finite fields.
\newblock {\em Memoirs of AMS}, to appear.

\bibitem{GKZ}
I.~M. Gelfand, M.~M. Kapranov, and A.~V. Zelevinsky.
\newblock Generalized {E}uler integrals and {$A$}-hypergeometric functions.
\newblock {\em Adv. Math.}, 84(2):255--271, 1990.

\bibitem{Goursat}
\'{E}douard Goursat.
\newblock Sur l'\'{e}quation diff\'{e}rentielle lin\'{e}aire, qui admet pour
  int\'{e}grale la s\'{e}rie hyperg\'{e}om\'{e}trique.
\newblock {\em Ann. Sci. \'{E}cole Norm. Sup. (2)}, 10:3--142, 1881.

\bibitem{Gr}
John Greene.
\newblock Hypergeometric functions over finite fields.
\newblock {\em Trans. Amer. Math. Soc.}, 301(1):77--101, 1987.

\bibitem{griff70}
Phillip~A. Griffiths.
\newblock Periods of integrals on algebraic manifolds: {S}ummary of main
  results and discussion of open problems.
\newblock {\em Bull. Amer. Math. Soc.}, 76:228--296, 1970.

\bibitem{GK}
Benedict~H. Gross and Neal Koblitz.
\newblock Gauss sums and the {$p$}-adic {$\Gamma $}-function.
\newblock {\em Ann. of Math. (2)}, 109(3):569--581, 1979.

\bibitem{Holz2}
R.-P. Holzapfel.
\newblock A voyage with three balloons.
\newblock {\em Math. Intelligencer}, 12(1):33--39, 1990.

\bibitem{Holz1}
Rolf-Peter Holzapfel.
\newblock {\em Geometry and arithmetic around {E}uler partial differential
  equations}, volume~11 of {\em Mathematics and its Applications (East European
  Series)}.
\newblock D. Reidel Publishing Co., Dordrecht, 1986.

\bibitem{Huber95}
Annette Huber.
\newblock {\em Mixed motives and their realization in derived categories},
  volume 1604 of {\em Lecture Notes in Mathematics}.
\newblock Springer-Verlag, Berlin, 1995.

\bibitem{Huber95ii}
Annette Huber.
\newblock Realization of {V}oevodsky's motives.
\newblock {\em J. Algebraic Geom.}, 9(4):755--799, 2000.

\bibitem{yI0}
Yasutaka Ihara.
\newblock Hecke {P}olynomials as congruence {$\zeta $} functions in elliptic
  modular case.
\newblock {\em Ann. of Math. (2)}, 85:267--295, 1967.

\bibitem{Iv1}
Florian Ivorra.
\newblock R\'{e}alisation {$l$}-adique des motifs triangul\'{e}s
  g\'{e}om\'{e}triques. {I}.
\newblock {\em Doc. Math.}, 12:607--671, 2007.

\bibitem{Iv2}
Florian Ivorra.
\newblock R\'{e}alisation {$\ell$}-adique des motifs triangul\'{e}s
  g\'{e}om\'{e}triques. {II}.
\newblock {\em Math. Z.}, 265(1):221--247, 2010.

\bibitem{Iv3}
Florian Ivorra.
\newblock Perverse, {H}odge and motivic realizations of \'{e}tale motives.
\newblock {\em Compos. Math.}, 152(6):1237--1285, 2016.

\bibitem{IvMor}
Florian {Ivorra} and Sophie {Morel}.
\newblock {The four operations on perverse motives}.
\newblock {\em arXiv e-prints}, page arXiv:1901.02096, January 2019.

\bibitem{Katz70i}
Nicholas~M. Katz.
\newblock Nilpotent connections and the monodromy theorem: {A}pplications of a
  result of {T}urrittin.
\newblock {\em Inst. Hautes \'{E}tudes Sci. Publ. Math.}, (39):175--232, 1970.

\bibitem{Katz72}
Nicholas~M. Katz.
\newblock Algebraic solutions of differential equations ({$p$}-curvature and
  the {H}odge filtration).
\newblock {\em Invent. Math.}, 18:1--118, 1972.

\bibitem{Katz88}
Nicholas~M. Katz.
\newblock {\em Gauss sums, {K}loosterman sums, and monodromy groups}, volume
  116 of {\em Annals of Mathematics Studies}.
\newblock Princeton University Press, Princeton, NJ, 1988.

\bibitem{Katz90}
Nicholas~M. Katz.
\newblock {\em Exponential Sums and Differential Equations}, volume 124 of {\em
  Annals of Mathematics Studies}.
\newblock Princeton University Press, Princeton, NJ, 1990.

\bibitem{Katz96}
Nicholas~M. Katz.
\newblock {\em Rigid local systems}, volume 139 of {\em Annals of Mathematics
  Studies}.
\newblock Princeton University Press, Princeton, NJ, 1996.

\bibitem{Katz09}
Nicholas~M. Katz.
\newblock Another look at the {D}work family.
\newblock In {\em Algebra, arithmetic, and geometry: in honor of {Y}u. {I}.
  {M}anin. {V}ol. {II}}, volume 270 of {\em Progr. Math.}, pages 89--126.
  Birkh\"auser Boston, Inc., Boston, MA, 2009.

\bibitem{koike-shiga1}
Kenji Koike and Hironori Shiga.
\newblock Isogeny formulas for the {P}icard modular form and a three terms
  arithmetic geometric mean.
\newblock {\em J. Number Theory}, 124(1):123--141, 2007.

\bibitem{Lau}
G.~Laumon.
\newblock Transformation de {F}ourier, constantes d'\'{e}quations
  fonctionnelles et conjecture de {W}eil.
\newblock {\em Inst. Hautes \'{E}tudes Sci. Publ. Math.}, (65):131--210, 1987.

\bibitem{morel-voevodsky}
Fabien Morel and Vladimir Voevodsky.
\newblock {${\bf A}^1$}-homotopy theory of schemes.
\newblock {\em Inst. Hautes \'{E}tudes Sci. Publ. Math.}, (90):45--143 (2001),
  1999.

\bibitem{picard}
\'{E}mile Picard.
\newblock Sur une extension aux fonctions de deux variables du probl\`eme de
  {R}iemann relatif aux fonctions hyperg\'{e}om\'{e}triques.
\newblock {\em Ann. Sci. \'{E}cole Norm. Sup. (2)}, 10:305--322, 1881.

\bibitem{picard2}
Emile Picard.
\newblock Sur des fonctions de deux variables ind\'{e}pendantes analogues aux
  fonctions modulaires.
\newblock {\em Acta Math.}, 2(1):114--135, 1883.

\bibitem{Shiga}
Hironori Shiga.
\newblock On the representation of the {P}icard modular function by {$\theta$}
  constants. {I}, {II}.
\newblock {\em Publ. Res. Inst. Math. Sci.}, 24(3):311--360, 1988.

\bibitem{Shioda1}
Tetsuji Shioda.
\newblock On rational points of the generic elliptic curve with level {$N$}
  structure over the field of modular functions of level {$N$}.
\newblock {\em J. Math. Soc. Japan}, 25:144--157, 1973.

\bibitem{St2}
P.~F. Stiller.
\newblock Classical automorphic forms and hypergeometric functions.
\newblock {\em J. Number Theory}, 28(2):219--232, 1988.

\bibitem{St0}
Peter~F. Stiller.
\newblock Differential equations associated with elliptic surfaces.
\newblock {\em J. Math. Soc. Japan}, 33(2):203--233, 1981.

\bibitem{St1}
Peter~F. Stiller.
\newblock A note on automorphic forms of weight one and weight three.
\newblock {\em Trans. Amer. Math. Soc.}, 291(2):503--518, 1985.

\bibitem{Tak}
Kisao Takeuchi.
\newblock Arithmetic triangle groups.
\newblock {\em J. Math. Soc. Japan}, 29(1):91--106, 1977.

\bibitem{Ter1}
Tomohide {Terasoma}.
\newblock {Hodge and Tate conjectures for hypergeometric sheaves}.
\newblock {\em arXiv e-prints}, pages alg--geom/9705023, August 1996.

\bibitem{TY1}
Fang-Ting Tu and Yifan Yang.
\newblock Algebraic transformations of hypergeometric functions and automorphic
  forms on {S}himura curves.
\newblock {\em Trans. Amer. Math. Soc.}, 365(12):6697--6729, 2013.

\bibitem{Vidunas}
Raimundas Vid\={u}nas.
\newblock Transformations of some {G}auss hypergeometric functions.
\newblock {\em J. Comput. Appl. Math.}, 178(1-2):473--487, 2005.

\bibitem{Weil1}
Andr\'{e} Weil.
\newblock Jacobi sums as ``{G}r\"{o}ssencharaktere''.
\newblock {\em Trans. Amer. Math. Soc.}, 73:487--495, 1952.

\bibitem{Weil2}
Andr\'{e} Weil.
\newblock Sommes de {J}acobi et caract\`eres de {H}ecke.
\newblock {\em Nachr. Akad. Wiss. G\"{o}ttingen Math.-Phys. Kl. II}, (1):1--14,
  1974.

\bibitem{Yu}
Jeng-Daw Yu.
\newblock Irregular {H}odge filtration on twisted de {R}ham cohomology.
\newblock {\em Manuscripta Math.}, 144(1-2):99--133, 2014.

\end{thebibliography}


\end{document}